\documentclass[12pt]{amsart}
\usepackage{amsmath,amsfonts,amssymb,amsthm}
\usepackage{anysize}
\marginsize{2cm}{2cm}{2cm}{2cm}
\usepackage{graphicx}
\usepackage{color}
\usepackage{soul}
\usepackage{enumerate}

\def\z{\mathfrak{z}}
\def\u{\mathfrak{u}}

\def\g{\mathfrak{g}}
\def\h{\mathfrak{h}}
\def\n{\mathfrak{n}}
\def\v{\mathfrak{v}}

\def\C{\mathbb{C}}
\def\R{\mathbb{R}}

\def\Z{\mathbb{Z}}
\def\N{\mathbb{N}}
\def\H{\mathbb H}

\def\ad{\operatorname{ad}}

\def\alt{\raise1pt\hbox{$\bigwedge$}}

\theoremstyle{plain}
\newtheorem{theorem}{\bf Theorem}[section]
\newtheorem{corollary}[theorem]{\bf Corollary}
\newtheorem{proposition}[theorem]{\bf Proposition}
\newtheorem{lemma}[theorem]{\bf Lemma}

\theoremstyle{definition}
\newtheorem{definition}[theorem]{\bf Definition}
\newtheorem{example}[theorem]{\bf Example}

\theoremstyle{remark}
\newtheorem{remark}[theorem]{\bf Remark}

\newcommand\aff{\mathfrak{aff}}

\begin{document}
	
\begin{abstract}
We introduce the notion of abelian almost contact structures on an odd dimensional real Lie algebra $\g$. This a sufficient condition for the structure to be normal. We investigate correspondences with even dimensional real Lie algebras endowed with an abelian complex structure, and with K\"ahler Lie algebras when $\g$ carries a compatible inner product. The classification of $5$-dimensional Sasakian Lie algebras with abelian structure is obtained. Later, we introduce and study abelian almost $3$-contact structures on real Lie algebras of dimension $4n+3$. These are given by triples of abelian almost contact structures, satisfying certain compatibility conditions, which are equivalent to the existence of a sphere of abelian almost contact structures. We obtain the classification of these Lie algebras in dimension $7$. Finally, we deal with the geometry of a Lie group $G$ endowed with a left invariant abelian almost $3$-contact structure and a compatible left invariant Riemannian metric.  We determine conditions for $G$ to admit a special metric connection with totally skew-symmetric torsion, called canonical, which plays the role of the Bismut connection for HKT structures arising from abelian hypercomplex structures. We provide examples and discuss the parallelism of the torsion of the canonical connection.
\end{abstract}

\title[Odd dimensional counterparts of abelian structures]{Odd dimensional counterparts of abelian complex and hypercomplex structures}
	
	
\author{Adri\'an Andrada}
\address[Andrada]{FaMAF-CIEM (CONICET), Universidad Nacional de C\'ordoba, Av. Medina Allende S/N, Ciudad Universitaria, X5000HUA
	C\'ordoba, Argentina} \email{andrada@famaf.unc.edu.ar}
	
\author{Giulia Dileo}
\address[Dileo]{Dipartimento di Matematica, Universit\`a degli Studi di Bari Aldo Moro, Via E. Orabona 4, 70125 Bari, Italy}
\email{giulia.dileo@uniba.it}


%

\thanks{This research was partially done during a visit of the first author at the Department of Mathematics of the University of Bari, financially supported by the research group GNSAGA of the Istituto Nazionale di Alta Matematica (Italy). The first author was also partially supported by CONICET, ANPCYT and SECYT-UNC (Argentina).}
	
\maketitle
	
\section{Introduction}

An abelian complex structure on a real Lie algebra $\g$ is an endomorphism $J$ of $\mathfrak g$ satisfying
	\begin{equation*} 
	J^2=-I, \hspace{1.5cm} [Jx,Jy]=[x,y] \quad	\forall x,y \in \g,
	\end{equation*}
where the second condition is equivalent to asking that the $i$-eigenspace $\g^{1,0}$ of $J$ be an abelian subalgebra of $\g^\C$. This implies the integrability of the corresponding left invariant almost complex structure $J$ defined on any Lie group $G$ with Lie algebra $\g$. It is known that every Lie algebra endowed with an abelian complex structure is $2$-step solvable (\cite{Pe}). The notion of abelian complex structure was introduced in \cite{BDM}, and afterwards developed in various directions. On the one hand one can consider a compatible inner product on the Lie algebra, giving rise to the notion of abelian Hermitian structures (\cite{ABD1}); on the other hand one can consider abelian hypercomplex structures, consisting of triples of abelian complex structures $\{J_i\}_{i=1,2,3}$ satisfying $J_1J_2=J_3=-J_2J_1$ (\cite{DF1,DF2}).

An interesting interplay with HKT geometry has been investigated in this context (\cite{DF,FG}). If a Lie algebra $\g$ is endowed with an abelian hyperHermitian structure $(\{J_i\}_{i=1,2,3},g)$, i.e. an abelian hypercomplex structure and a compatible inner product, the corresponding left invariant hyperHermitian structure on any Lie group $G$ with Lie algebra $\g$, makes $G$ a hyperK\"ahler with torsion manifold. This means that $G$ admits a (unique) metric connection $\nabla$ with totally skew-symmetric torsion such that $\nabla J_i=0$, $i=1,2,3$. In fact, this connection coincides with the Bismut connection of each Hermitian structure $(J_i,g)$, $i=1,2,3$, also called KT connection by physicists, whose torsion is the $3$-form $c=J_id\Omega_i$; here $\Omega_i=g(\cdot,J_i\cdot)$ denotes the K\"ahler form for $J_i$.
Further, the HKT structure arising from an abelian hyperHermitian structure turns out to be weak, in the sense that $dc\ne 0$.

\

The aim of the present paper is to provide and study odd dimensional counterparts of abelian complex and hypercomplex structures, whose natural setting is given by almost contact and almost $3$-contact structures. We will introduce the relative notions of \emph{abelian} structures on Lie algebras, taking also into account the left invariant structures defined on the corresponding Lie groups. In particular, on a Lie group $G$ endowed with an abelian almost $3$-contact structure, considering a compatible left invariant Riemannian metric $g$, we will look for a \lq good' metric connection with totally skew-symmetric torsion, playing the role of the Bismut connection in HKT geometry.

\

An \emph{almost contact structure} on a real Lie algebra ${\mathfrak g}$ is a triple $(\varphi,\xi,\eta)$ with $\varphi\in \mathrm{End}({\mathfrak g})$, $\xi\in{\mathfrak g}$, $\eta\in{\mathfrak g}^*$ satisfying
	\[\varphi^2=-I+\eta\otimes \xi,\quad \eta(\xi)=1,\]
which imply that $\varphi(\xi)=0$ and $\eta\circ\varphi=0$. The Lie algebra splits as $\mathfrak g=\mathbb{R}\xi\oplus\mathfrak h $, where the subspace $\mathfrak h:=\mathrm{Ker}\,\eta=\mathrm{Im}\,\varphi$ is endowed with the complex structure $J:=\varphi|_{\mathfrak h}:\mathfrak h\to{\mathfrak h}$. We will say that the almost contact structure is \emph{abelian} if
\[\mathrm{ad}_\xi\circ\varphi=\varphi\circ \mathrm{ad}_\xi\quad \mbox{ and }\quad [\varphi X,\varphi Y]=[X,Y]\quad \forall X,Y\in\mathfrak h.\]
This is equivalent to asking that in the complexified Lie algebra ${\mathfrak g}^{\mathbb{C}}={\mathbb{C}\xi \oplus \mathfrak h}^{1,0}\oplus{\mathfrak h}^{0,1}$ the $i$-eigenspace $\h^{1,0}$ of $J$ be an $\mathrm{ad}_\xi$-invariant abelian subalgebra. Further, this is a sufficient condition for the almost contact structure to be \emph{normal}.

Examples of Lie algebras endowed with an abelian almost contact structure $(\varphi,\xi,\eta)$ can be obtained as extensions of a Lie algebra $\h$, with abelian complex structure $J$, in the following two different ways:
\begin{itemize}
\item[-] considering the $1$-dimensional central extension of $\h$ by a $J$-invariant $2$-cocycle; in fact, this is a characterization of abelian almost contact structures with central $\xi$, and this is the case of the real Heisenberg Lie algebra $\h_n^\R$;
\item[-] considering the $1$-dimensional extension of $\h$ by a derivation $D:\h\to\h$  commuting with $J$, which characterizes abelian almost contact structures with closed $1$-form $\eta$.
    \end{itemize}
In both cases, the almost contact Lie algebra obtained is solvable.

 In the $3$-dimensional case, an almost contact structure is abelian if and only if $\mathrm{ad}_\xi\circ\varphi=\varphi\circ \mathrm{ad}_\xi$, and the classification of these Lie algebras can be easily obtained (Proposition \ref{dim3}).

A natural further step consists in considering almost contact metric Lie algebras, that is Lie algebras endowed with an almost contact structure $(\varphi,\xi,\eta)$ and a compatible inner product $g$, which satisfies
\[g(\varphi X,\varphi Y)=g(X,Y)-\eta(X)\eta(Y)\qquad\forall X,Y\in\g\]
so that $\h=\langle\xi\rangle^\perp$. It is well known that the real Heisenberg Lie algebra is endowed with a compatible inner product making $\h_n^\R$ a Sasakian Lie algebra. We will examine correspondences with K\"ahler Lie algebras $(\h,J,g)$, with abelian complex structure, analyzing the two types of $1$-dimensional extensions of $\h$ mentioned above. We will also classify $5$-dimensional Sasakian Lie algebras with abelian structure (Section \ref{dim5}).

\

In Section \ref{section-3contact} we introduce \emph{abelian almost $3$-contact structures}. These are defined as triples of abelian almost contact structures $\{(\varphi_i,\xi_i,\eta_i)\}_{i=1,2,3}$, such that
\begin{equation}
		\begin{split}\label{3-sasaki1}
		\varphi_k=\varphi_i\varphi_j-\eta_j\otimes\xi_i=-\varphi_j\varphi_i+\eta_i\otimes\xi_j,\quad\\
		\xi_k=\varphi_i\xi_j=-\varphi_j\xi_i, \quad
		\eta_k=\eta_i\circ\varphi_j=-\eta_j\circ\varphi_i,
		\end{split}
	\end{equation}
for every even permutation $(i,j,k)$ of $(1,2,3)$. In general, the existence of an almost $3$-contact structure is equivalent to the existence of a sphere $\{(\varphi_a,\xi_a,\eta_a)\}_{a\in S^2}$ of almost contact structures such that
\[\varphi_a\circ\varphi_b-\eta_b\otimes\xi_a=\varphi_{a\times b}-(a\cdot b)\,I,\qquad \varphi_a\xi_b=\xi_{a\times b},\qquad \eta_a\circ\varphi_b=\eta_{a\times b},\]
for every $a,b\in S^2$, where $\cdot$ and $\times$ denote the standard inner product and cross product on $\mathbb{R}^3$.
We show that if the three structures $(\varphi_i,\xi_i,\eta_i)$, $i=1,2,3$, satisfying \eqref{3-sasaki1} are abelian, then every structure in the sphere is abelian.

A Lie algebra ${\mathfrak g}$ endowed with an almost $3$-contact structure splits as ${\mathfrak g}={\mathfrak h}\oplus{\mathfrak v}$, as direct sum of vector spaces, called respectively	\emph{horizontal} and
	\emph{vertical}, and defined by \[
	{\mathfrak h} \, :=\, \bigcap_{i=1}^{3}\operatorname{Ker}\eta_i,\qquad
	{\mathfrak v}\, :=\, \langle\xi_1,\xi_2,\xi_3\rangle.
	\]
	In particular $\dim{\mathfrak h}=4n$. It is interesting to note that, if the structure is abelian, many features of the Lie algebra $\g$ are encoded by a horizontal vector ${\mathcal Z}\in\h$ and an endomorphism $\psi:\h\to\h$ defined by
\[{\mathcal Z}:=\varphi_i([\xi_j,\xi_k]), \qquad \psi:=(\mathrm{ad}_{\xi_i}\circ\varphi_i)|_{\mathfrak h}=(\varphi_i\circ\mathrm{ad}_{\xi_i})|_{\mathfrak h},\]
for every even permutation $(i,j,k)$ if $(1,2,3)$. They satisfy
\begin{equation*}
	[\xi_j,\xi_k]=-\varphi_i{\mathcal Z}+2\delta\xi_i,\qquad \psi({\mathcal Z})=0,\qquad \operatorname{ad}_{\mathcal Z}|_{\mathfrak h}=2(\psi^2+\delta \psi),
	\end{equation*}
for some real number $\delta$. In particular,
\begin{itemize}
\item[-] ${\mathcal Z}=0$ if and and only if $\v$ is a subalgebra of $\g$, in which case either $\mathfrak v$ is abelian or isomorphic to $\mathfrak{so}(3)$, according to $\delta=0$ or $\delta\ne 0$;
\item[-] $\mathcal Z$ is central if and only if $\psi^2+\delta \psi=0$.
\end{itemize}
The rank of $\psi$ is $4p$, $0\leq p\leq n$, and we are able to further describe the Lie algebra when $\psi$ has minimum or maximum rank, according to the following scheme:
\begin{itemize}
\item $\psi=0$, ${\mathcal Z}=0$, $\delta=0$ ($\Leftrightarrow \v\subset\z(\g)$):\smallskip\\
 the horizontal subspace $\mathfrak h$ admits a Lie algebra structure and an abelian hypercomplex structure $\{J_1,J_2,J_3\}$, such that $\mathfrak g\cong\R^3\oplus_{\theta} \h$ is the central extension of  $\mathfrak h$ by a $J_i$-invariant $\mathfrak v$-valued $2$-cocycle $\theta$ of $\mathfrak h$;\medskip
\item $\psi=0$, ${\mathcal Z}=0$, $\delta\ne0$:\smallskip\\
$\h$ and $\v$ are ideals of $\g$, and $\g\cong \mathfrak{so}(3)\times \h$, $\h$ carrying an abelian hypercomplex structure;

\medskip
\item $\psi=0$, ${\mathcal Z}\ne0$:\smallskip\\
 $\h$ is an ideal of $\g$; it carries an abelian hypercomplex structure and its center contains $\mathcal Z$ and $\mathcal Z_i:=\varphi_i \mathcal Z$ for $i=1,2,3$.
The adjoint action of $\xi_i$, is given by
\[
 [\xi_i,\xi_j]=2\delta\xi_k-\mathcal Z_k, \qquad [\xi_i,X]=0,
\]
for any even permutation $(i,j,k)$ of $(1,2,3)$ and for any $X \in\h$. If $\delta\neq 0$ then $\g\cong \u\times \h$, where $\u=\text{span}\{2\delta \xi_1-\mathcal{Z}_1,2\delta \xi_2-\mathcal{Z}_2, 2\delta \xi_3-\mathcal{Z}_3 \}$ is an ideal isomorphic to $\mathfrak{so}(3)$;
\medskip
\item $\psi$ invertible ($\Rightarrow {\mathcal Z}=0,\; \delta\ne0$):\smallskip\\
$\h$ is an abelian ideal, and the non zero brackets are given by
	\[
	[\xi_i,\xi_j]=2\delta \xi_k, \qquad [\xi_i,X]= \delta \varphi_i X,
	\]
	for any even permutation $(i,j,k)$ of $(1,2,3)$,  $X\in\mathfrak h$; therefore $\g\cong\mathfrak{so}(3)\ltimes \mathbb R^{4n}$.
\end{itemize}

\

The information in the above scheme turns out to be particularly helpful in those cases where either $\psi=0$ or $\psi$ is invertible.
This is the case of $7$-dimensional almost $3$-contact Lie algebras with abelian structure, which in fact we completely classify (Theorem \ref{theo7dim}). A second case is given by almost $3$-contact Lie algebras of dimension $4n+3$ endowed with a \emph{canonical abelian} structure.

\

\emph{Canonical abelian almost $3$-contact structures} are defined in Section \ref{section-canonical}. In order to clarify the geometric motivation in introducing such a notion, it is worth spending a few words on the role of connections with totally skew-symmetric torsion (skew torsion for short) in Riemannian geometry. For details we refer to \cite{Ag}. On a Riemannian manifold $(M,g)$ one can distinguish $8$ classes of geometric torsion tensors for a metric connection. Among them metric connections with skew torsion are one of the most studied. They have the same geodesics as the Levi-Civita connection and arise naturally in large classes of manifolds. We have already mentioned the case of KT and HKT structures in the context of Hermitian and hyperHermitian geometry, where the Bismut connection is now a well-established tool. Moreover, there is a vast class of manifolds admitting a metric connection $\nabla$ with $\nabla$-parallel skew torsion, in which case interesting geometric properties are satisfied: the curvature tensor of the connection is pair symmetric and the $\nabla$-Ricci tensor is symmetric (\cite{IP,CS,CMS}).

Recently, metric connections with totally skew-symmetric torsion appeared in the investigation of various classes of almost $3$-contact metric manifolds (\cite{AF,AFS,AgDi,DOP}). By an almost $3$-contact metric manifold we mean a differentiable manifold endowed with a triple of almost contact structures $(\varphi_i,\xi_i,\eta_i)$, $i=1,2,3$, satisfying the same equations as in \eqref{3-sasaki1}, and a compatible Riemannian metric. For these manifolds the Levi-Civita connection is not well adapted to their structure, essentially because they do not appear in Berger's theorem on irreducible Riemannian holonomies. On the other hand, if one looks for \lq good' metric connections with skew torsion, the existence of a connection parallelizing all the structure tensor fields is not guaranteed. $3$-Sasakian manifolds, which are the most famous class of almost $3$-contact metric manifolds, do not admit such a connection.

In \cite{AgDi} the class of \emph{canonical} almost $3$-contact metric manifolds has been introduced, and characterized as those manifolds admitting a (unique) metric connection $\nabla$ with totally skew-symmetric torsion such that
\[\nabla_X\varphi_i\, =\, \beta(\eta_k(X)\varphi _j -\eta_j(X)\varphi _k)\]
for every even permutation $(i,j,k)$ of $(1,2,3)$, and vector field $X$, and for some smooth function $\beta$. The connection $\nabla$ is called \emph{canonical}, and also satisfies
\[ \nabla_X\xi_i\, =\, \beta(\eta_k(X)\xi_j -\eta_j(X)\xi_k), \qquad \nabla_X\eta_i\, =\, \beta(\eta_k(X)\eta_j -\eta_j(X)\eta_k).\]
 When $\beta$ vanishes, $\nabla$ parallelizes all the structure tensors, in which cases the structure is called \emph{parallel canonical}.

 \

Now, considering a Lie algebra $\g$ with abelian almost $3$-contact structure $(\varphi_i,\xi_i,\eta_i)$, if $G$ is any Lie group with Lie algebra $\g$, endowed with the corresponding left invariant structure $(\varphi_i,\xi_i,\eta_i)$ and a compatible left invariant Riemannian metric $g$, it is natural to require the Lie group $(G,\varphi_i,\xi_i,\eta_i,g)$ to be canonical in the sense of \cite{AgDi}, thus admitting a canonical connection. We show that this is equivalent to requiring
\[{\mathcal Z}=0,\qquad \psi=-\frac{\beta}{2} I,\]
where $\beta\in\R$. Therefore, for a canonical abelian almost $3$-contact metric structure, we have either $\psi=0$ or $\psi$ invertible, according to $\beta=0$ or $\beta\ne0$, which correspond to the parallel and non-parallel case respectively. In both cases we further describe the Lie algebras, provide examples, and discuss the parallelism of the torsion of the canonical connection.

In the parallel case, remarkable examples of Lie groups endowed with a canonical abelian structure are
\[H_n^{\H},\qquad H_{n}^{\mathbb C}\times \mathbb R,\qquad H_{2n}^{\mathbb R}\times {\mathbb R}^2,\]
where $H_n^{\H}, H_{n}^{\mathbb C}, H_{2n}^{\mathbb R}$  are respectively the quaternionic, the complex and the real Heisenberg group, of real dimensions $4n+3$, $4n+2$, $4n+1$. In all these cases the Lie algebra is the central extension of the abelian Lie algebra $\R^{4n}$ by a $J_i$-invariant $\R^3$-valued $2$-cocycle, where $\{J_i\}_{i=1,2,3}$ is the standard hypercomplex structure on $\R^{4n}$. In the non-parallel case, every Lie group admitting a left invariant canonical abelian structure is isomorphic to the semidirect product $\operatorname{SU}(2)\ltimes \H^n$. We show that this Lie group admits co-compact discrete subgroups, so that the associated compact quotients also admit structures of the same type.

\medskip

\

\noindent \textsc{Acknowledgements.} The authors would like to thank I. Agricola, M. L. Barberis, I. Dotti and A. Tolcachier for helpful comments. The first author is very grateful to the Dipartimento di Matematica at the Universit\`a degli Studi di Bari Aldo Moro for the warm hospitality during his visit.

\

\section{Preliminaries}

\subsection{Abelian complex and hypercomplex structures on Lie algebras}
	
A complex structure on a real Lie algebra $\mathfrak g$ is an endomorphism $J$ of $\mathfrak g$  satisfying $J^2=-I$ and such that
   \begin{equation}\label{nijen}
	J[x,y]-[Jx,y]-[x,Jy]-J[Jx,Jy]=0, \quad \forall\, x,y \in \mathfrak g.
	\end{equation}
It is well known that \eqref{nijen} holds if and only if ${\mathfrak g}^{1,0}$, the $i$-eigenspace of $J$, is a complex subalgebra of $\mathfrak g ^{\mathbb C} := {\mathfrak g}\otimes_\mathbb R \mathbb C$. This endomorphism $J$ of $\g$ gives rise to an integrable almost complex structure on any Lie group $G$ with Lie algebra $\g$, such that the left translations are holomorphic maps of $G$.
	
An \textit{abelian} complex structure on $\mathfrak g$ is an endomorphism $J$ of $\mathfrak g$ satisfying
	\begin{equation} \label{abel1}
	J^2=-I, \hspace{1.5cm} [Jx,Jy]=[x,y], \quad	\forall x,y \in \g.
	\end{equation}
It follows that \eqref{abel1} is a particular case of \eqref{nijen}; moreover, condition \eqref{abel1} is equivalent to
$\mathfrak g^{1,0}$ being abelian. These structures were first considered in \cite{BDM}.
	
Next, we include some properties about abelian complex structures in the following lemma (see \cite{ABD,BD,Pe} for their proofs).
	
\begin{lemma}\label{lemma-J}
Let $\mathfrak g$ be a Lie algebra with $\mathfrak{z}(\mathfrak g)$ its center and $\mathfrak g':=[\mathfrak g,\mathfrak g]$ its commutator ideal. If $J$ is an abelian complex structure on $\mathfrak g$, then
	\begin{enumerate}
			\item[\textsc{1.}] $J\mathfrak z(\mathfrak g)=\mathfrak z(\mathfrak g)$.
			\item[\textsc{2.}] $\mathfrak g' \cap J\mathfrak g'\subset \mathfrak z(\mathfrak g'+J\mathfrak g')$.
			\item[\textsc{3.}] The codimension of $\mathfrak g'$ is at least $2$, unless $\mathfrak g$ is isomorphic to
			$\mathfrak{aff}(\mathbb R)$ (the only $2$-di\-men\-sional non-abelian Lie algebra).
			\item[\textsc{4.}] $\mathfrak g'$ is abelian, therefore $\mathfrak g$ is $2$-step solvable.
	\end{enumerate}
\end{lemma}
	
\

If $g$ is an inner product on $\g$ which is compatible with an abelian complex structure $J$, then $(J,g)$ is called an \textit{abelian Hermitian structure} on $\g$.  In \cite{ABD1} many properties of abelian Hermitian structures were established. We mention the following result concerning K\"ahler structures, which will be used in forthcoming sections. Recall that a Hermitian structure is called \textit{K\"ahler} if the fundamental $2$-form $\Omega=g(\cdot,J\cdot)$ is closed.
	
\begin{theorem}\cite[Theorem 4.1]{ABD1}\label{kahler-abelian}
	Let $(\mathfrak g, J, g)$ be a K\"ahler Lie algebra with $J$ an abelian complex structure. Then $\mathfrak g$ is isomorphic to
	\[ \mathfrak{aff}(\mathbb R)\times \cdots \times \mathfrak{aff}(\mathbb R)\times \mathbb{R}^{2s},\]
	and this decomposition is orthogonal and $J$-stable.
\end{theorem}
	
Moreover, it follows from the proof of this theorem that, if there are $k$ factors isomorphic to $\aff(\R)$, then there exists an orthonormal basis $\{e_1,f_1,\ldots,e_k,f_k\}$ of $\aff(\R)^k$ and real numbers $r_i\neq 0$, $i=1,\ldots,k$, such that $[e_i,f_i]=r_if_i$ and $Je_i=f_i$ for all $i$.

\

We recall now another generalization of abelian complex structures. A hypercomplex structure on a Lie algebra $\g$ is a triple $\{J_1,J_2,J_3\}$ of complex structures on $\g$ which obey the laws of the quaternions:
\[  J_1J_2=-J_2J_1=J_3.  \]
If $\g$ admits such a structure, then the dimension of $\g$ is a multiple of $4$. Moreover, $\g$ carries a sphere $S^2$ of complex structures. Indeed, if $a=(a_1,a_2,a_3)\in S^2$ then
\begin{equation}\label{J-sphere}
	J_a:=a_1J_1+a_2J_2+a_3J_3
\end{equation}
is a complex structure on $\g$.

\smallskip

A hypercomplex structure $\{J_1,J_2,J_3\}$ on $\g$ is called \textit{abelian} if each complex structure $J_i$, $i=1,2,3$, is abelian. In this case, it can be seen that each complex structure in the associated sphere is abelian. Indeed, a stronger result was proved in \cite{DF2}: if $\{J_1,J_2,J_3\}$ is a hypercomplex structure on $\g$ and one complex structure in the associated sphere is abelian, then all the complex structures in the sphere are abelian.

\smallskip

The classification of hypercomplex structures on $4$-dimensional Lie algebras was carried out in \cite{Bar}. In particular, we can deduce from this classification the following result concerning abelian hypercomplex structures:

\begin{proposition}\cite{Bar}\label{hypercomplex-4d}
	The only $4$-dimensional Lie algebras which admit an abelian hypercomplex structure are the abelian one, i.e. $\R^4$, and $\aff(\C)$.  In the first case the abelian hypercomplex structure is unique (up to hypercomplex isomorphism), whereas in the second case any hypercomplex structure is abelian and the equivalence classes of abelian hypercomplex structures are parametrized by $\R P^2$.
\end{proposition}

Recall that $\aff(\C)$ is the Lie algebra with basis $\{e_1,e_2,e_3,e_4\}$ whose Lie bracket is given by
\[ [e_1,e_3]=e_3, \; [e_1,e_4]=e_4, \; [e_2,e_3]=e_4,\; [e_2,e_4]=-e_3.  \]
It follows from the proof of Theorem 3.3 in \cite{Bar} that any hypercomplex structure on $\aff(\C)$ is equivalent to $\{J_a,J_b,J_{a\times b}\}$ (as in \eqref{J-sphere}), for some $a,b\in S^2$ such that $a\perp b$, where $\{J_1,J_2,J_3\}$ is given in the basis above by
	\begin{equation}\label{aff-C}
	J_1=\begin{pmatrix} & 1 && \\ -1 &&& \\ &&& -1\\ && 1 & \end{pmatrix}, \quad J_2=\begin{pmatrix} && -1 & \\  &&& -1 \\ 1 &&& \\ & 1 && \end{pmatrix}, \quad J_3=\begin{pmatrix} &&& -1 \\ && 1 & \\ & -1 && \\ 1 &&& \end{pmatrix}.
	\end{equation}
Therefore, the spheres determined by $\{J_a,J_b,J_{a\times b}\}$  and $\{J_1,J_2,J_3\}$ coincide.

\

In forthcoming sections we will have to cope with central extensions of Lie algebras equipped with abelian complex or hypercomplex structures. We recall this notion here. Let $\h$ be a Lie algebra and $V$ a finite-dimensional vector space. A $V$-valued $p$-form on $\h$ is an element $\sigma\in \alt^p\h^\ast \otimes V$, that is, a skew-symmetric multilinear function $\sigma:\h\times \cdots \times \h \to V$. The exterior derivative $d$ on $\h$ can be extended to $V$-valued $p$-forms in the following way: if $\{e_i\}$ is a basis of $V$, then $\sigma$ can be written as $\sigma=\sum_i \sigma_i\otimes e_i$, with $\sigma_i\in\alt^p\h^\ast$, and $d\sigma\in \alt^{p+1}\h^\ast \otimes V$ is defined as $d\sigma=\sum_i d\sigma_i\otimes e_i$. This definition does not depend on the chosen basis. If $d\sigma=0$, then $\sigma$ is called a $V$-valued $p$-cocycle. Note that when $\dim V=1$ we recover the usual notion of $p$-forms and exterior derivative.

Let us fix now a $V$-valued $2$-cocycle $\sigma$. On the vector space $V\oplus \h$ we consider the bracket $[\cdot,\cdot]'$ defined by:
\[  [V,V\oplus \h]'=0, \qquad [X,Y]'=  \sigma(X,Y)+[X,Y], \quad X,Y\in\h.  \]
It follows from $d\sigma=0$ that this bracket satisfies the Jacobi identity; clearly, $V$ is a central subalgebra. The Lie algebra $(V\oplus \h,[\cdot,\cdot]')$ is called the central extension of $\h$ by the $V$-valued $2$-cocycle $\sigma$, and it is denoted $V\oplus_{\sigma}\h$. In this article we will deal only with $V=\R$ or $V=\R^3$.
	
\

\subsection{Almost contact structures}
	
An \emph{almost contact structure} on a differentiable manifold $M^{2n+1}$ is a triple $(\varphi,\xi,\eta)$, where $\varphi$ is a $(1,1)$-type tensor field, $\xi$ a vector field, and $\eta$ a $1$-form satisfying
	\[\varphi^2=-I+\eta\otimes\xi, \quad\eta(\xi)=1,\]
which imply $\varphi(\xi)=0$ and $\eta\circ\varphi=0$. The tangent bundle splits as $TM^{2n+1}={\mathcal D}\oplus{\mathcal L}$, where $\mathcal D=\mathrm{Ker}\,\eta=\mathrm{Im}\,\varphi$ and ${\mathcal L}$ is the line bundle spanned by $\xi$.
	
On the product manifold $M^{2n+1}\times \mathbb{R}$ one can define an almost complex structure $J$ by
	\[J\left(X,f\frac{d}{dt}\right)=\left(\varphi
	X-f\xi,\eta(X)\frac{d}{dt}\right),\]
where $X$ is a vector field tangent to $M^{2n+1}$, $t$ is the coordinate on $\mathbb{R}$ and $f$ is a smooth function on $M^{2n+1}\times \mathbb{R}$. If $J$ is integrable, the almost contact structure is said to be \emph{normal}. This is equivalent to the
vanishing of the tensor field
	\begin{equation}\label{normality-tensor}
	N_\varphi:=[\varphi,\varphi]+d\eta\otimes\xi,
	\end{equation}
where $[\varphi,\varphi]$ is the Nijenhuis torsion of $\varphi$ defined by
\[ [\varphi,\varphi](X,Y)=[\varphi X,\varphi Y]+\varphi^2[X,Y]-\varphi[\varphi X,Y]-\varphi[X,\varphi Y].\]
	
An \emph{almost contact metric structure} $(\varphi, \xi, \eta, g)$	is given by an almost contact structure and a \emph{compatible}
Riemannian metric $g$ satisfying $g(\varphi X,\varphi Y)=g(X,Y)-\eta(X)\eta(Y)$ for any vector fields $X$ and $Y$. Then, the \emph{fundamental $2$-form} $\Phi$ is defined by $\Phi(X,Y)=g(X,\varphi Y)$ for any vector fields $X$ and $Y$. For more details, we refer to \cite{BLAIR}.

\
	
An \emph{almost contact Lie algebra} is a $(2n+1)$-dimensional real Lie algebra ${\mathfrak g}$ endowed with a triple $(\varphi,\xi,\eta)$ with $\varphi\in \mathrm{End}({\mathfrak g})$, $\xi\in{\mathfrak g}$, $\eta\in{\mathfrak g}^*$ satisfying
	\[\varphi^2=-I+\eta\otimes \xi,\quad \eta(\xi)=1,\]
implying that $\varphi(\xi)=0$ and $\eta\circ\varphi=0$. The Lie algebra splits as $\mathfrak g=\mathfrak h\oplus \mathbb{R}\xi$, where $\mathfrak h:=\mathrm{Ker}\,\eta=\mathrm{Im}\,\varphi$. The structure is said to be \emph{normal} if the tensor $N_\varphi$ defined as in \eqref{normality-tensor} vanishes on $\mathfrak g$. An \emph{almost contact metric Lie algebra} $({\mathfrak g},\varphi,\xi,\eta,g)$ is an almost contact Lie algebra endowed with an inner product $g$ satisfying $g(\varphi X,\varphi Y)=g(X,Y)-\eta(X)\eta(Y)$, so that ${\mathfrak h}=\langle\xi\rangle^\perp$. The fundamental $2$-form is defined by $\Phi(X,Y)=g(X,\varphi Y)$. Since in this paper we will be interested in a special class of normal almost contact Lie algebras, we recall some remarkable classes of them. A normal almost contact metric Lie algebra is called
		\begin{itemize}
			\item[-] \emph{$\alpha$-Sasakian} if $d\eta=2\alpha\Phi$ for some $\alpha\in\mathbb{R}^*$, Sasakian for $\alpha=1$;
			\item[-] \emph{coK\"ahler} if $d\eta=0$ and $d\Phi=0$;
			\item[-] \emph{$\alpha$-Kenmotsu} if $d\eta=0$, $d\Phi=2\alpha\eta\wedge\Phi$ for some $\alpha\in\mathbb{R}^*$, Kenmotsu for $\alpha=1$;
			\item[-] \emph{quasi-Sasakian} if $d\Phi=0$ (this class includes $\alpha$-Sasakian and coK\"ahler structures).
		\end{itemize}
Regarding coK\"ahler and $\alpha$-Kenmotsu structures, some other terminologies can be found in the literature. For instance,
coK\"ahler structures are also known as cosymplectic structures (\cite{FV}). In \cite{CP} a normal almost contact metric structure satisfying $d\eta=0$ and $d\Phi=2\alpha\eta\wedge\Phi$, for some $\alpha\in\mathbb{R}$, is called \emph{$\alpha$-coK\"ahler}.

\

\subsection{Almost $3$-contact structures}
An  \emph{almost $3$-contact structure}  on a differentiable manifold $M^{4n+3}$ is given by three almost contact structures $(\varphi_i,\xi_i,\eta_i)$, $i=1,2,3$,
satisfying
\begin{equation*}
\begin{split}
\varphi_k=\varphi_i\varphi_j-\eta_j\otimes\xi_i=-\varphi_j\varphi_i+\eta_i\otimes\xi_j,\quad\\
\xi_k=\varphi_i\xi_j=-\varphi_j\xi_i, \quad
\eta_k=\eta_i\circ\varphi_j=-\eta_j\circ\varphi_i,
\end{split}
\end{equation*}
for any even permutation $(i,j,k)$ of $(1,2,3)$ (\cite{BLAIR}). The tangent bundle of $M^{4n+3}$ splits as $TM^{4n+3}={\mathcal H}\oplus{\mathcal V}$, where
\[
{\mathcal H} \, :=\, \bigcap_{i=1}^{3}\operatorname{Ker}\eta_i,\qquad
{\mathcal V}\, :=\, \langle\xi_1,\xi_2,\xi_3\rangle.
\]
In particular ${\mathcal H}$ has rank $4n$. We call any vector belonging to the distribution
${\mathcal H}$ \emph{horizontal} and any vector belonging to the distribution $\mathcal V$
\emph{vertical}. The manifold is said to  be
\emph{hypernormal} if each  almost contact structure
$(\varphi_i,\xi_i,\eta_i)$ is normal.
In \cite{yano1} it was proved that if two of the almost contact structures are normal,
then so is the third.

  The existence of an almost $3$-contact structure is equivalent to the existence of a sphere of almost contact structures $\{(\varphi_a,\xi_a,\eta_a) \ |\ a\in S^2 \}$ satisfying
\[
\varphi_a\circ\varphi_b-\eta_b\otimes\xi_a=\varphi_{a\times b}-(a\cdot b)\,I,\qquad
\varphi_a\xi_b=\xi_{a\times b}, \qquad
\eta_a\circ\varphi_b=\eta_{a\times b},
\]
for every $a,b\in S^2$, where $\cdot$ and $\times$ denote the standard inner product and cross product on $\mathbb{R}^3$.  If the structure is hypernormal, then each structure in the sphere is normal (\cite{CM-DN-Y}).	

Any almost $3$-contact manifold admits a Riemannian metric $g$ which is compatible with
each of the three structures. Then $M^{4n+3}$ is said to be an \emph{almost $3$-contact metric
manifold} with structure $(\varphi_i,\xi_i,\eta_i,g)$, $i=1,2,3$.
The subbundles $\mathcal H$
and $\mathcal V$ are orthogonal with respect to $g$ and the three Reeb vector
fields $\xi_1,\xi_2,\xi_3$ are orthonormal. The structure group of the tangent
bundle is in fact reducible to $\mathrm{Sp}(n)\times \{1\}$.

We will introduce analogous notions on Lie algebras in Section \ref{section-3contact}.

\
	
\section{Abelian almost contact structures on Lie algebras}

In this section we introduce the main notion of this article, namely abelian almost contact structures on Lie algebras (Definition \ref{definition-abelian}), and begin the study of their main properties, together with some examples.

\medskip

We start by analyzing when an almost contact structure on a Lie algebra is normal.

\begin{proposition}\label{proposition-normality}
An almost contact structure $(\varphi,\xi,\eta)$ on a Lie algebra ${\mathfrak g}$ is normal if and only if
		\begin{itemize}
			\item[\textsc{1.}] $\mathrm{ad}_\xi\circ\varphi=\varphi\circ \mathrm{ad}_\xi$,
			\item[\textsc{2.}] $[\varphi X,\varphi Y]-[X,Y]=\varphi([\varphi X,Y]+[X,\varphi Y])$ for every $X,Y\in\mathfrak h$.
		\end{itemize}
If the structure is normal, the subspace $\mathfrak h$ is $\mathrm{ad}_\xi$-invariant, i.e. $\eta([\xi,X])=0$ for any $X\in\mathfrak h$. Furthermore, for every $X,Y\in\h$
		\begin{equation}\label{normal-deta}
		d\eta(\varphi X,\varphi Y)=d\eta(X,Y).
		\end{equation}
\end{proposition}

\begin{proof}
From  the definition of $N_\varphi$ it follows that for every $X,Y\in{\mathfrak h}$
		\begin{align*}
		N_\varphi(X,Y)&=[\varphi X,\varphi Y]-[X,Y]-\varphi([\varphi X,Y]+[X,\varphi Y]),\\
		N_\varphi(\xi,X)&={}-[\xi,X]-\varphi[\xi,\varphi X],
		\end{align*}
which immediately give the equivalence. The fact that $\mathfrak h$ is $\mathrm{ad}_\xi$-invariant is consequence of \textsc{1.} Finally, equation \eqref{normal-deta} follows from $d\eta(\xi,X)=0$ and by applying the $1$-form $\eta$ on both sides in \textsc{2.}
\end{proof}

\begin{definition}\label{definition-abelian}
An almost contact structure $(\varphi,\xi,\eta)$ on a Lie algebra ${\mathfrak g}$, will be called \emph{abelian} if
		\begin{itemize}
			\item[\textsc{1.}] $\mathrm{ad}_\xi\circ\varphi=\varphi\circ \mathrm{ad}_\xi$,
			\item[\textsc{2.}] $[\varphi X,\varphi Y]=[X,Y]$ for every $X,Y\in\mathfrak h$.
		\end{itemize}
\end{definition}

Condition \textsc{2.} is equivalent to
	\[ [\varphi X,Y]+[X,\varphi Y]=0\qquad \forall X,Y\in\mathfrak h,\]
and by Proposition \ref{proposition-normality}, it follows that an abelian almost contact structure is normal.
	
Now, given an almost contact Lie algebra $({\mathfrak g},\varphi,\xi,\eta)$, the endomorphism $J:=\varphi|_{\mathfrak h}:\mathfrak h\to{\mathfrak h}$ satisfies $J^2=-I$. Then the complexification ${\mathfrak g}^{\mathbb{C}}={\mathfrak g}\oplus i{\mathfrak g}$ splits as
\begin{equation}\label{g-C}
\g^{\mathbb{C}}={\mathbb{C}\xi \oplus \mathfrak h}^{1,0}\oplus{\mathfrak h}^{0,1}
\end{equation}
where
\[{\mathfrak h}^{1,0}=\{X-iJX\,|\,X\in{\mathfrak h}\},\qquad {\mathfrak h}^{0,1}=\{X+iJX\,|\,X\in{\mathfrak h}\}\]
are the eigenspaces of $J:{\mathfrak h}^{\mathbb{C}}\to {\mathfrak h}^{\mathbb{C}}$ corresponding to the eigenvalues $i$ and $-i$ respectively.
	
\begin{proposition}\label{h10}
Let $({\mathfrak g},\varphi,\xi,\eta)$ be an almost contact Lie algebra. Then,
	\begin{itemize}
		\item[\textsc{1}.] the structure is normal if and only if ${\mathfrak h}^{1,0}$ is an $\mathrm{ad}_\xi$-invariant subalgebra of ${\mathfrak g}^\mathbb{C}$;
		\item[\textsc{2.}] the structure is abelian if and only if ${\mathfrak h}^{1,0}$ is an $\mathrm{ad}_\xi$-invariant abelian subalgebra of ${\mathfrak g}^\mathbb{C}$.
	\end{itemize}
\end{proposition}

\begin{proof}
Both equivalences are consequence of the following identities:
		\begin{align*}
		[X-iJX,Y-iJY]&=[X,Y]-[JX,JY]-i([X,JY]+[JX,Y]),\\
		[\xi,X-iJX]&=[\xi,X]-i[\xi,JX]
		\end{align*}
for every $X,Y\in\mathfrak h$.
\end{proof}

\

\subsection{Abelian almost contact structures with $\xi\in{\mathfrak z}({\mathfrak g})$}
	
We will investigate almost contact Lie algebras $(\g,\varphi,\xi,\eta)$ with abelian structure  for which $\mathrm{ad}_\xi=0$, that is $\xi\in{\mathfrak z}({\mathfrak g})$. A first known example is the following.

\begin{example}\label{Heisenberg}
	Let ${\mathfrak g}={\mathfrak h}_{n}^{\mathbb{R}}$ be the real Heisenberg Lie algebra of dimension $2n+1$. It admits a basis $\{\xi,X_1,\ldots,X_n,Y_1,\ldots,Y_n\}$, with non vanishing commutators
	\[[X_i,Y_i]=2\xi,\qquad i=1,\ldots,n.\]
	Let $\varphi\in \mathrm{End}({\mathfrak g})$ and $\eta\in{\mathfrak g}^*$ such that
	\[\varphi\xi=0,\quad\varphi X_i=Y_i,\quad\varphi Y_i=-X_i,\quad \eta(\xi)=1,\quad \eta(X_i)=\eta(Y_i)=0,\]
	for every $i=1,\ldots,n$. Then $(\varphi,\xi,\eta)$ is an abelian almost contact structure on $\mathfrak g$. A compatible inner product on $\mathfrak g$ can be defined requiring the basis to be orthonormal. Then $(\mathfrak g,\varphi,\xi,\eta,g)$ is a Sasakian Lie algebra with abelian structure.
\end{example}

\

In our next result we exhibit an analogy with Lie algebras equipped with abelian complex structures (see Lemma \ref{lemma-J}).

\begin{lemma}\label{2step}
	Let $(\g,\varphi,\xi,\eta)$ be an almost contact Lie algebra with abelian structure such that $\xi\in{\mathfrak z}(\g)$. Then $\g$ is $2$-step solvable.	
\end{lemma}

\begin{proof}
	According to Proposition \ref{h10}, the subspaces $\h^{1,0}$ and $\h^{0,1}$ are abelian subalgebras of $\g^\C$. Moreover, since $\xi$ is central, $\C\xi\oplus \h^{1,0}$ is an abelian subalgebra, and it follows from \eqref{g-C} that $\g^\C$ can be written as a direct sum (as vector spaces) of two abelian subalgebras, $\C\xi\oplus \h^{1,0}$ and $\h^{0,1}$. Thus, according to \cite{Pe} (see also \cite[Proposition 2.3]{ABDO}), $\g^\C$ is $2$-step solvable, hence $\g$ is $2$-step solvable.
\end{proof}

\

In the following result we show that an almost contact Lie algebra with abelian structure such that $\xi\in\z(\g)$ can be written as a central extension of a Lie algebra with an abelian complex structure, via a $J$-invariant $2$-cocycle.

\begin{proposition}\label{proposition-extension}
	Let $({\mathfrak g},\varphi,\xi,\eta)$ be an almost contact Lie algebra with abelian structure such that $\xi\in{\mathfrak z}({\mathfrak g})$. Then $\mathfrak h=\operatorname{Ker}\eta$ admits a Lie algebra structure and an abelian complex structure given by $J=\varphi|_{\mathfrak h}$, such that $\mathfrak g$ is the $1$-dimensional central extension of $\h$  by a $J$-invariant $2$-cocycle $\sigma$. Moreover, the cocycle $\sigma$ is given by $\sigma(X,Y)=-d\eta(X,Y)$, $X,Y\in\h$.
		
	Conversely, if $\mathfrak h$ is a Lie algebra equipped with an abelian complex structure $J$ and $\sigma$ is a $J$-invariant $2$-cocycle
	on $\h$, then the central extension  $\g=\R\xi\oplus_\sigma{\h}$  carries a natural abelian almost contact structure $(\varphi,\xi,\eta)$.
		
	Moreover, $\eta$ is a contact form if and only if $\sigma$ is of maximal rank.
\end{proposition}

\begin{proof}
We consider the decomposition $\g=\R\xi\oplus \h$ as a direct sum of vector spaces. Accordingly we can write, for any $X,Y\in\h$,
	\begin{equation} \label{decomposition}
		[X,Y]=\sigma(X,Y)\xi+[X,Y]_\h,
	\end{equation}
where $\sigma(X,Y)\in\R$ and  $[X,Y]_\h\in\h$. If $X,Y,Z\in \h$ then taking components in $[[X,Y],Z]+[[Y,Z],X]+[[Z,X],Y]=0$ and using that $\xi$ is central, we arrive at
		\begin{align}\label{jacobi}
		[[X,Y]_\h,Z]_\h+[[Y,Z]_\h,X]_\h+[[Z,X]_\h,Y]_\h & =0,\\
		\sigma([X,Y]_\h,Z)+\sigma([Y,Z]_\h,X)+\sigma([Z,X]_\h,Y) & =0.\nonumber
		\end{align}
The first equation in \eqref{jacobi} is the Jacobi identity for the bracket $[\cdot,\cdot]_\h$ on  $\h$. Then the second equation in \eqref{jacobi} means that $\sigma$ is a $2$-cocycle on $(\h,[\cdot,\cdot]_\h)$. From this it follows that $\g$ is the central extension of $(\h,[\cdot,\cdot]_\h)$ by the $2$-cocycle $\sigma$.
		
It is clear that $J$ satisfies $J^2=-I$ on $\h$. Moreover, using \eqref{decomposition}, it follows  from $[\varphi X,\varphi Y]=[X,Y]$ that
		\[ [JX,JY]_\h=[X,Y]_\h, \qquad \sigma(JX,JY)=\sigma(X,Y),\]
so that $J$ is an abelian complex structure on $\h$ and $\sigma$ is a $J$-invariant cocycle on $\h$. By applying $\eta$ in both sides of \eqref{decomposition} we arrive at $\sigma(X,Y)=-d\eta(X,Y)$, $X,Y\in\h$.
		
As for the converse, it is enough to set $\varphi|_\h=J$, $\varphi(\xi)=0$ and $\eta|_\h=0$, $\eta(\xi)=1$.
		
The last statement is clear.
\end{proof}
	
\medskip

As a consequence of Lemma \ref{2step} and Proposition \ref{proposition-extension} we obtain

\begin{corollary}
	Let $J$ be an abelian complex structure on a Lie algebra $\h$. If $\sigma\in\alt^2\h^*$ is a $J$-invariant $2$-cocycle, then $\sigma(x,y)=0$ for $x,y\in\h'$.
\end{corollary}

\medskip


\begin{example}
Let $\mathfrak h$ be a Lie algebra equipped with an abelian complex structure $J$. Then any exact $2$-form $\sigma=df$, with
	$f\in\mathfrak{h}^*$, is $J$-invariant, since
	\[df(JX,JY)=-f([JX,JY])=-f([X,Y])=df(X,Y),\]
	for any $X,Y\in\mathfrak h$. However, in this case the almost contact Lie algebra $\g=\R \xi\oplus_\sigma \h$ from Proposition \ref{proposition-extension} is isomorphic to the direct product $\R\times \h$. Indeed, the linear map $$T:\mathfrak g\to \mathbb R \times \mathfrak h,\qquad T(a\xi+X)=(a+f(X))\xi+X,\quad a\in\mathbb R, X\in\mathfrak h$$
	is a Lie algebra isomorphism. In particular, on the subalgebra $\mathfrak{k}=T^{-1}({\mathfrak h})=\{-f(X)\xi+X\,|\,
	X\in{\mathfrak h}\}$,
	one can define the abelian complex structure $J'$ by $J'(-f(X)\xi+X)=-f(JX)\xi+JX$, which satisfies $T(J'Z)=J(TZ)$ for every $Z\in{\mathfrak k}$.
\end{example}

\

Now, let us consider an almost contact metric Lie algebra $(\g,\varphi,\xi,\eta,g)$ with abelian structure such that $\xi$ is central. By Proposition \ref{proposition-extension}, $\g$ is the $1$-dimensional central extension of a Hermitian Lie algebra $(\h,J,g)$, with abelian complex structure, by a $J$-invariant $2$-cocycle $\sigma$. Then, using \eqref{decomposition} and the fact that $\xi$ is central, it is easy to verify that the structure $(\varphi,\xi,\eta,g)$  is quasi-Sasakian, i.e. $d\Phi=0$, if and only if $(\h,J,g)$ is a K\"ahler Lie algebra with abelian complex structure $J$.
	
Consequently, by Theorem \ref{kahler-abelian}, any quasi-Sasakian Lie algebra $(\g,\varphi,\xi,\eta,g)$ with abelian structure and such that $\xi\in{\z}({\g})$, is isomorphic to
	\[\R\xi\oplus_{\sigma}\left(\aff(\R)\times \cdots \times \aff(\R)\times \R^{2s}\right).\]
Further, the structure is $\alpha$-Sasakian if and only if $\sigma=-2\alpha\omega$, where $\omega$ is the K\"ahler form of $(\h,J,g)$.

\medskip
	
\begin{example}\label{example-quasi}
Let us consider the abelian Lie algebra $\h=\R^{2n}$ spanned by $X_i,Y_i$, $i=1,\ldots,n$, with K\"ahler structure $(J,g)$ such that $JX_i=Y_i$, and the basis is orthonormal. Denote by $X^i$, $Y^i$, $i=1,\ldots,n$ the dual basis of ${\h}^*$. For any $k\leq n$ the $J$-invariant $2$-cocycle $\sigma_k=2\sum_{i=1}^k X^i\wedge Y^i$ gives rise to a Lie algebra $\g_k=\R\xi\oplus_{\sigma_k} \h$ endowed with a quasi-Sasakian abelian structure. One can easily verify that $\g_k$ is isomorphic to $\h_k^{\R}\times {\R}^{2(n-k)}$. When $k=n$, that is $\sigma_n=-2\omega$, where $\omega$ is the K\"ahler form of $(\h,J,g)$, we obtain the Sasakian structure on $\h_n^{\R}$ as in Example \ref{Heisenberg}. Notice that $\sigma_k$ is not exact for any $k$.
\end{example}

\

In the following we provide one further example with a non-exact $2$-cocycle, and where furthermore $\mathfrak h$ is a non-abelian algebra.
	
\begin{example}\label{affC}
Let $\mathfrak h:=\mathfrak{aff}(\mathbb C)$ be the $4$-dimensional Lie algebra introduced in Proposition \ref{hypercomplex-4d}. This Lie algebra admits many abelian complex structures, since it admits an abelian hypercomplex one. Let us choose the abelian complex structure  $J=J_1$ in \eqref{aff-C}, that is, $Je_1=-e_2, \, Je_3=e_4$. The differential $d:\mathfrak h^*\to \alt^2\mathfrak h^*$ is given by
	\begin{equation}\label{d-affC}
	de^1=0,\; de^2=0, \; de^3=-e^{13}+e^{24}, \; de^4=-e^{14}-e^{23}.
	\end{equation}
Hence the $2$-form $\sigma$ given by $\sigma=e^{12}$ satisfies $d\sigma=0$ and $\sigma(J\cdot,J\cdot)=\sigma(\cdot,\cdot)$, so that $\sigma$ is a $J$-invariant $2$-cocycle on $\h$. Clearly $\sigma$ is not exact. The central extension $\g=\R \xi\oplus_\sigma \h$ has Lie brackets given by:
	\[ [e_1,e_2]=\xi, \;[e_1,e_3]=e_3, \; [e_1,e_4]=e_4, \; [e_2,e_3]=e_4,\; [e_2,e_4]=-e_3.  \]
Notice that $\g$ is not isomorphic to $\R\times \aff(\C)$.
		
Taking the abelian almost contact structure $(\varphi,\xi,\eta)$ on $\g$ defined in Proposition \ref{proposition-extension} and any compatible inner product $g$, the structure $(\varphi,\xi,\eta,g)$ cannot be quasi-Sasakian since $\aff(\C)$ does not admit any abelian K\"ahler structure.
\end{example}

\

\subsection{Abelian almost contact structures with $d\eta=0$}
	
We study now almost contact Lie algebras $(\g,\varphi,\xi,\eta)$ with abelian structure, for which the subspace $\h=\operatorname{Ker} \eta$ is a subalgebra, or equivalently $d\eta=0$. Therefore, $J:=\varphi|_{\h}$ is an abelian complex structure on $\h$.

\begin{proposition}\label{prop-extension-D}
Let $(\g,\varphi,\xi,\eta)$ be an almost contact Lie algebra with abelian structure, such that $\h=\operatorname{Ker} \eta$ is a subalgebra. Then $\mathfrak g$ is the semidirect product $\g=\R\xi\ltimes_D \h$, where the derivation $D:=\ad_{\xi}:\h\to\h$ commutes with the abelian complex structure $J=\varphi\,|_{\h}$.

Conversely, if $\h$ is any Lie algebra endowed with an abelian complex structure $J$ and a derivation $D$ commuting with $J$, then $\g:=\R \ltimes_D \h$ admits an abelian almost contact structure.
\end{proposition}

\begin{proof}
The first part is an immediate consequence of the fact that the subalgebra $\mathfrak h$ is $\mathrm{ad}_{\xi}$-invariant. As for the second part, let $\mathfrak h$ be a Lie algebra with abelian complex structure $J$, and let $D:{\mathfrak h}\to{\mathfrak h}$ be a derivation such that $DJ=JD$. We consider the Lie algebra $\mathfrak g={\mathbb R}\xi\ltimes {\mathfrak h}$, where $\mathfrak h$ is a subalgebra and $[\xi,X]:=DX$ for every $X\in\mathfrak h$. Then $\mathfrak g$ can be endowed with an abelian almost contact structure $(\varphi,\xi,\eta)$ defined by $\varphi\xi=0$, $\varphi X=JX$, 	$\eta(\xi)=1$, $\eta(X)=0$ for every $X\in\mathfrak h$.
\end{proof}
	
\

\begin{corollary}
	Let $(\g,\varphi,\xi,\eta)$ be an almost contact Lie algebra with abelian structure, such that $\h=\operatorname{Ker} \eta$ is a subalgebra. Then $\g$ is solvable.
\end{corollary}

\begin{proof}
	According to Proposition \ref{prop-extension-D} we have that $\g=\R\xi \ltimes_D \h$ for a certain derivation $D$ of $\h$. Since both $\g/\h\cong \R$ and $\h$ are solvable (see Lemma \ref{lemma-J}), it follows that $\g$ is solvable as well.
\end{proof}

\smallskip

\begin{example}
	The nilpotent Lie algebra $\h:=\h_1^\R\times \R$ admits an abelian complex structure $J$, and it is easy to find an invertible derivation $D$ of $\h$ which commutes with $J$. Therefore the Lie algebra $\g:=\R\ltimes_{D} \h$ has commutator ideal $\g'=\h$, which is not abelian. Thus $\g$ is not $2$-step solvable.
\end{example}

\

Recall that a remarkable class of almost contact metric structures with closed $1$-form $\eta$ is given by $\alpha$-coK\"ahler structures. In the following, considering a compatible inner product, we characterize $\alpha$-coK\"ahler Lie algebras with abelian structure.

\begin{proposition}\label{prop-alpha-coKahler}
Let $({\mathfrak g},\varphi,\xi,\eta, g)$ be an almost contact metric Lie algebra with abelian structure. Then $(\varphi,\xi,\eta, g)$ is an $\alpha$-coK\"ahler structure if and only if
	\begin{itemize}
		\item [\textsc{1.}] ${\mathfrak h}=\operatorname{Ker} \eta$ is a subalgebra of $\mathfrak g$ endowed with an abelian K\"ahler structure;
		\item[\textsc{2.}] the symmetric part of the derivation $\mathrm{ad}_\xi:{\mathfrak h}\to \mathfrak h$ coincides with $-\alpha I$.
	\end{itemize}
\end{proposition}

\begin{proof}
Assume that $(\varphi,\xi,\eta, g)$ is $\alpha$-coK\"ahler , that is  $d\eta=0$ and $d\Phi=2\alpha\eta\wedge\Phi$, $\alpha\in\R$. Then ${\mathfrak h}=\operatorname{Ker} \eta$ is a subalgebra of $\mathfrak g$, endowed with the abelian Hermitian structure $(J, g)$. This structure is K\"ahler since $d\Phi(X,Y,Z)=0$ for every $X,Y,Z\in\mathfrak h$. Now, taking $X,Y\in\mathfrak h$, we have
	\begin{align*}
	d\Phi(\xi,X,Y)&=-\Phi([\xi,X],Y)-\Phi([X,Y],\xi)-\Phi([Y,\xi],X)\\
	              &=-g(\mathrm{ad}_\xi X,JY)+g(\mathrm{ad}_\xi Y,JX)\\
		          &=-g(\mathrm{ad}_\xi X,JY)-g(\mathrm{ad}_\xi JY,X)\\
		          &=-g(\mathrm{ad}_\xi X,JY)-g(\mathrm{ad}_\xi^*X, JY)
	\end{align*}
where we used the fact that $\mathrm{ad}_\xi:{\mathfrak h}\to \mathfrak h$ commutes with $J$. Being $d\Phi(\xi,X,Y)=2\alpha\Phi(X,Y)=2\alpha g(X,JY)$, we have that the symmetric part of $\mathrm{ad}_\xi$ coincides with $-\alpha I$. The converse follows analogously.
\end{proof}

\medskip
	
In particular, as a consequence of the proposition above and Theorem \ref{kahler-abelian}, any $\alpha$-coK\"ahler Lie algebra $(\g,\varphi,\xi,\eta,g)$ with abelian structure is isomorphic to
\[\R\xi\ltimes_{D}\left(\aff(\R)\times \cdots \times \aff(\R)\times \R^{2s}\right),\]
where $D$ is a derivation of $\aff(\R)^k\times \R^{2s}$ that commutes with $J$ and whose symmetric part is $-\alpha I$. It is easy to verify that such a derivation $D$ is given by
\[ D|_{\aff(\R)^k}=-\alpha I,\qquad D|_{\R^{2s}}=-\alpha I+A, \qquad \text{with } A\in\u(s). \]

In particular, if the structure is coK\"ahler, i.e. $\alpha=0$,  then $\g$ is isomorphic to
\[\aff(\R)^k\times (\R\xi\ltimes_A \R^{2s})\]
for some $A\in\u(s)$. Considering the associated simply connected Lie group with left invariant metric, the factor corresponding to $\R\xi\ltimes_A \R^{2s}$ is flat (according to \cite{Mi}), while each factor corresponding to $\aff(\R)$ has negative constant curvature.

\

\begin{example}
Consider the Lie algebra $\g=\R\xi\ltimes_D \R^{2n}$, where $\R^{2n}$ is equipped with its standard K\"ahler structure, and $D=-\alpha I$ for some $\alpha\neq 0$. Then $\g$ admits an abelian almost contact metric structure $(\varphi, \xi,\eta,g)$ which is $\alpha$-Kenmotsu. Notice that if $G$ is any Lie group with Lie algebra $\mathfrak g$, endowed with the corresponding left invariant $\alpha$-Kenmotsu structure, then $G$ is locally isometric to the hyperbolic space of constant curvature $K=-\alpha^2$ (see \cite{D,Mi}).
\end{example}
	
\

\begin{remark}
For the one-to-one correspondence between $\alpha$-coK\"ahler Lie algebras of dimension $2n+1$ and K\"ahler Lie algebras of dimension $2n$ see also \cite{FV, CP}.
\end{remark}
	
\	
	
\subsection{Classification in dimension $3$}
	
Let $({\mathfrak g},\varphi,\xi,\eta)$ be a 3-dimensional almost contact Lie algebra with abelian structure. Note that condition $\textsc{2.}$ from Definition \ref{definition-abelian} holds trivially on such a Lie algebra, so that the fact that this structure is abelian is equivalent to $[\mathrm{ad}_\xi,\varphi]=0$.
	
Let us assume that $\g$ is not abelian. There exists a basis $\{e_1,e_2\}$ of $\h:=\operatorname{Ker} \eta$ such that $e_2=\varphi(e_1),\, e_1=-\varphi(e_2)$. In this basis, the operator $\mathrm{ad}_\xi|_\h$ takes the following form:
	\[ \mathrm{ad}_\xi|_{\h}=\begin{pmatrix} a & -b \\ b & a \end{pmatrix},\]
for some $a,b\in\mathbb R$. Therefore, the Lie bracket on $\mathfrak g$ is described by
\begin{align*}
	[\xi,e_1] &= ae_1+be_2,\\
	[\xi,e_2] &= -be_1+ae_2,\\
	[e_1,e_2] &= \alpha \xi+\beta e_1+\gamma e_2,
\end{align*}
for some $\alpha,\beta,\gamma\in\mathbb R$. The Jacobi identity in this case is equivalent to the following system of equations:
	\begin{equation}\label{system}
	\begin{cases}
	a\alpha=0,\\
	a\beta+b\gamma=0,\\
	a\gamma-b\beta=0.
	\end{cases}
	\end{equation}
		
\textsl{Case 1}: $\xi$ is central. In this case we have $a=b=0$ and the system \eqref{system} is trivially satisfied. The commutator ideal has dimension 1 and hence $\g$ is isomorphic either to $\h_1^\R$ or $\aff(\R)\times \R$ (see for instance \cite{BD}). The former case happens only when $\beta^2+\gamma ^2= 0$, while the latter case occurs only when $\beta^2+\gamma ^2\neq 0$. Note that $\h$ is a Lie subalgebra if and only if $\alpha=0$.
	
\medskip
	
\textsl{Case 2}: $\xi$ is not central. In this case we have $a^2+b^2\neq 0$, and it follows from \eqref{system} that $\beta=\gamma=0$. If $\alpha\neq 0$, we have that $a=0$, $b\neq 0$,  and the equations in \eqref{system} become
	\[ [\xi,e_1] = be_2, \quad [\xi,e_2]= -be_1, \quad [e_1,e_2]= \alpha \xi.\]
Therefore $\mathfrak g$ is isomorphic to $\mathfrak{so}(3)$ if $b\alpha>0$, whereas $\mathfrak g$ is isomorphic to $\mathfrak{sl}(2,\mathbb R)$ if $b\alpha<0$ (see \cite{Mi}). Moreover, $\eta$ is a contact form.
	
If $\alpha=0$ then $\h$ is a subalgebra of $\mathfrak g$ and hence $\eta$ is not a contact form on $\mathfrak g$. The equations in \eqref{system} become
	\[ [e_1,e_2]=0,\quad [\xi,e_1]=ae_1+be_2, \quad [\xi,e_2]=-be_1+ae_2.\]
If $b=0$ (and thus $a\neq 0$) then $\mathfrak g$ is a non-unimodular completely solvable Lie algebra, isomorphic to $\mathfrak{r}_{3,1}$, in the notation from \cite{ABDO}. On the other hand, if $b\neq 0$ then $\mathfrak g$ is a solvable Lie algebra which is not completely solvable and it is unimodular if and only if $a=0$. Moreover, it is isomorphic to $\mathfrak{r}'_{3,\frac{a}{b}}$ in the notation from \cite{ABDO}. A Lie algebra with parameters $a,b$ is isomorphic to another with parameters $a',b'$ (with both $b$ and $b'$ different from zero) if and only if $\frac{a'}{b'}=\pm \frac{a}{b}$.

\medskip

Summarizing, we have proved the following result

\begin{proposition}\label{dim3}
Let $({\mathfrak g},\varphi,\xi,\eta)$ be a $3$-dimensional almost contact Lie algebra with abelian structure. Then $\g$ is isomorphic to:
\[ \R^3,\quad \h_1^\R, \quad \mathfrak{aff}(\mathbb R)\times\mathbb R, \quad \mathfrak{so}(3),\quad \mathfrak{sl}(2,\R), \quad \mathfrak{r}_{3,1},\quad \mathfrak{r}'_{3,\lambda},    \]
with $\lambda\geq 0$.
\end{proposition}

\

\begin{remark}
Let $(\mathfrak{g}, \varphi,\xi,\eta,g)$ be a $3$-dimensional almost contact metric Lie algebra with abelian structure. Then, choosing an orthonormal basis $\{\xi,e_1,e_2\}$ such that $\varphi e_1=e_2$, the fundamental $2$-form of the structure is $\Phi=-e^{12}$, where $\{\eta,e^1,e^2\}$ is the dual basis of $\mathfrak{g}^*$. Following the notations above, we have
\begin{align*}
d\eta &= -\alpha e^{12},\\
de^1 & = -a \eta\wedge e^1+b \eta\wedge e^2-\beta e^{12},\\
de^2 & = -b \eta\wedge e^1-a \eta\wedge e^2-\gamma e^{12},
\end{align*}
and thus, we get
\[d\Phi=-de^1\wedge e^2+e^1\wedge de^2=2a\eta\wedge e^1\wedge e^2=-2a\eta\wedge\Phi.\]
Therefore, if $\alpha\ne 0$ (in which case $a=0)$, we get $d\eta=\alpha\Phi$ and the structure is $\frac{\alpha}{2}$-Sasakian. If $\alpha=0$, we have $d\eta=0$ and $d\Phi=-2a\eta\wedge\Phi$. In this case, if $a=0$ the structure is coK\"ahler, if $a\ne0$ the structure is $(-a)$-Kenmotsu.
\end{remark}

\medskip

We point out that every 3-dimensional Sasakian Lie algebra has abelian structure, since in this case we always have $[\ad_\xi,\varphi]=0$. It is known that any 3-dimensional Sasakian Lie algebra is isomorphic to either $\h_1^\R$, $\mathfrak{aff}(\mathbb R)\times\mathbb R$, $\mathfrak{so}(3)$ or $\mathfrak{sl}(2,\R)$ (see \cite{P}).

\

\subsection{5-dimensional Sasakian Lie algebras with abelian structure}\label{dim5}
In \cite{AFV} the classification of 5-dimensional Sasakian Lie algebras was provided. This classification was divided into two cases, corresponding to non-trivial center (necessarily of dimension one) or trivial center.

Our purpose is to determine all 5-dimensional Sasakian Lie algebras with abelian structure.

We consider first the case when the center is not trivial. It follows from the discussion previous to Example \ref{example-quasi} that if $\g$ is a 5-dimensional Sasakian Lie algebra with abelian structure and non trivial center then there is an isomorphism $\g=\R\xi\oplus_{-2\omega} \h$, where $\h$ is isomorphic to $\R^4$, $\aff(\R)\times \R^2$ or $\aff(\R)\times\aff(\R)$, each of them equipped with a K\"ahler structure with abelian complex structure, and $\omega$ is the corresponding K\"ahler form. We analyze each case separately; in all of them $\{e_1,e_2,e_3,e_4\}$ denotes an orthonormal basis such that $Je_1=e_2,\, Je_3=e_4$.
\begin{itemize}
	\item When $\h\cong \R^4$, then $\g$ is isomorphic to $\h_2^\R$, with the Sasakian structure from Example \ref{Heisenberg}.
	\item When $\h\cong \aff(\R)\times \R^2$,  the Lie bracket on $\g$ is given by $[e_1,e_2]=re_2+2\xi, \, [e_3,e_4]=2\xi$, for some $r\neq 0$.
	\item When $\h\cong \aff(\R)\times \aff(\R)$, the Lie bracket on $\g$ is given by $[e_1,e_2]=re_2+2\xi, \, [e_3,e_4]=se_4+2\xi$, for some $r,s\neq 0$.
\end{itemize}

\medskip

On the other hand, in the case when the center is trivial, we follow the proof of the classification given in \cite[Section 3.2]{AFV} and we add the equations imposed by condition \textsc{2.} in Definition \ref{definition-abelian}. It follows that abelian Sasakian structures occur only in cases (B3) and (B4) of that classification and any Lie algebra $\g$ which admits such a structure is given by the following equations, where $\{e_1,\ldots,e_5\}$ denotes an orthonormal basis of $\g$ and $\{e^1, \ldots,e^5\}$ is the dual basis:
\begin{align*}
de^1 & =  2\cos\theta\, (e^{12}+e^{34}), \\
de^2 & =  2\sin\theta\, (e^{12}+e^{34}), \\
de^3 & =  e^{45}+\sin\theta\, (e^{13}+e^{24}) +\cos\theta\, (e^{14}-e^{23}),\\
de^4 & =  -e^{35}-\cos\theta\, (e^{13}+e^{24}) +\sin\theta\, (e^{14}-e^{23}),\\
de^5 & =  2(e^{12}+e^{34}),
\end{align*}
for some $\theta\in[0,2\pi)$. Here the abelian Sasakian structure is given by
\[ \xi=e_5,\qquad \eta=e^5,\qquad \varphi(e_1)=-e_2,\qquad \varphi(e_3)=-e_4.\]
 It was shown in \cite{AFV} that for any value of $\theta$ the Lie algebra obtained is isomorphic to a Lie algebra denoted $\g_0$, which is non-unimodular and solvable. Therefore we obtain a $S^1$-family of abelian Sasakian structures on the Lie algebra $\g_0$. Note that this Lie algebra appears numbered 22 in the classification of 5-dimensional solvable Lie algebras with contact structure given in \cite{Di1}.

\
	
\section{Abelian almost $3$-contact structures}\label{section-3contact}

In this section we introduce the notion of abelian almost $3$-contact structures. We describe the main properties of any Lie algebra endowed with such a structure, showing that the Lie algebra carries a sphere of abelian almost contact structure. We also describe examples.
	
\begin{definition}
An \emph{almost $3$-contact Lie algebra} is a $(4n+3)$-dimensional real Lie algebra ${\mathfrak g}$ endowed with three almost contact structures $(\varphi_i,\xi_i,\eta_i)$ such that
	\begin{equation}
		\begin{split}\label{3-sasaki}
		\varphi_k=\varphi_i\varphi_j-\eta_j\otimes\xi_i=-\varphi_j\varphi_i+\eta_i\otimes\xi_j,\quad\\
		\xi_k=\varphi_i\xi_j=-\varphi_j\xi_i, \quad
		\eta_k=\eta_i\circ\varphi_j=-\eta_j\circ\varphi_i,
		\end{split}
	\end{equation}
for any even permutation $(i,j,k)$ of $(1,2,3)$. The Lie algebra $({\mathfrak g},\varphi_i,\xi_i,\eta_i)$ is said to be hypernormal if $N_{\varphi_i}=0$ for every $i=1,2,3$.
\end{definition}
	
\begin{definition}
	An \emph{almost $3$-contact metric Lie algebra} $({\mathfrak g},\varphi_i,\xi_i,\eta_i, g)$ is an almost $3$-contact Lie algebra endowed with a compatible inner product, that is an inner product $g$ satisfying
\begin{equation*}
g(\varphi_iX,\varphi_iY)=g(X,Y)-\eta_i(X)\eta_i(Y)
\end{equation*}
		for every $X,Y\in{\mathfrak g}$ and for every $i=1,2,3$.
\end{definition}
An almost $3$-contact Lie algebra ${\mathfrak g}$ splits as ${\mathfrak g}={\mathfrak h}\oplus{\mathfrak v}$, as direct sum of vector spaces, where
	\[
	{\mathfrak h} \, :=\, \bigcap_{i=1}^{3}\operatorname{Ker}\eta_i,\qquad
	{\mathfrak v}\, :=\, \langle\xi_1,\xi_2,\xi_3\rangle.
	\]
In particular $\dim{\mathfrak h}=4n$. We call any vector belonging to ${\mathfrak h}$ \emph{horizontal} and any vector belonging to ${\mathfrak v}$	\emph{vertical}. If $g$ is a compatible inner product, then $\mathfrak h$ and $\mathfrak v$ are orthogonal with respect to $g$ and the vectors $\xi_1$, $\xi_2$, $\xi_3$ are orthonormal.

\begin{definition}
	An almost $3$-contact structure $(\varphi_i,\xi_i,\eta_i)$ on a Lie algebra $\g$ will be called \textit{abelian} if each structure $(\varphi_i,\xi_i,\eta_i)$, $i=1,2,3$, is abelian.
\end{definition}

Obviously, every abelian almost $3$-contact structure on a Lie algebra is hypernormal.

\

\subsection{General properties of Lie algebras with abelian almost 3-contact structures}

We establish some fundamental properties of abelian almost 3-contact structures on Lie algebras. In fact, we show the existence of a vector $\mathcal Z\in \h$ and an endomorphism $\psi:\h\to\h$ which encode many features of the Lie algebra and the abelian structure.

\

Given an almost $3$-contact Lie algebra $({\mathfrak g},\varphi_i,\xi_i,\eta_i)$, we put $\zeta_k:=[\xi_i,\xi_j]$ for every even permutation $(i,j,k)$ of $(1,2,3)$. As a first result we prove the following lemma.
	
\begin{lemma}\label{lemma1}
	Let $({\mathfrak g},\varphi_i,\xi_i,\eta_i)$ be an almost $3$-contact Lie algebra with abelian structure. Then, for every  even permutation $(i,j,k)$ of $(1,2,3)$,
		\begin{itemize}
			\item[\textsc{1.}] $\varphi_i\zeta_j=\zeta_k=-\varphi_j\zeta_i\;$ and $\;\varphi_i\zeta_i=\varphi_j\zeta_j$;
			\item[\textsc{2.}] $\zeta_k\in\operatorname{Ker}\eta_i\cap\operatorname{Ker}\eta_j\;$ and $\;\eta_i\zeta_i=\eta_j\zeta_j$.
		\end{itemize}
\end{lemma}

\begin{proof}
Since $\mathrm{ad}_{\xi_i}\circ\varphi_i=\varphi_i\circ\mathrm{ad}_{\xi_i}$, we have
		\[ \varphi_i\zeta_j=\varphi_i[\xi_k,\xi_i]=[\varphi_i\xi_k,\xi_i]=[-\xi_j,\xi_i]=\zeta_k,\]
which implies that $\zeta_k\in\operatorname{Ker}\eta_i$. Analogously, one shows that $\varphi_j\zeta_i=-\zeta_k\in\operatorname{Ker}\eta_j$.
Consequently, we have
		\begin{align*}
		\varphi_j\zeta_j&=\varphi_j(\varphi_k\zeta_i)=\varphi_i\zeta_i+\eta_k(\zeta_i)\xi_j=\varphi_i\zeta_i,\\
		\eta_j\zeta_j&=\eta_j(\varphi_k\zeta_i)=\eta_i\zeta_i,
		\end{align*}
		thus completing the proof.
	\end{proof}

\
	
As a consequence of 1. in Lemma \ref{lemma1}, for an abelian almost $3$-contact structure we can define
\begin{equation*}\label{def-zeta}
{\mathcal Z}:=\varphi_i\zeta_i, \qquad i=1,2,3.
\end{equation*}
Note that ${\mathcal Z}\in{\mathfrak h}$. We also remark that property 2. in Lemma \ref{lemma1} can be equivalently expressed by
	\[\eta_r([\xi_s,\xi_t])=2\delta\epsilon_{rst}\]
for some real number $\delta$ and for every $r,s,t=1,2,3$, where $\epsilon_{rst}$ denotes the totally skew-symmetric symbol. Therefore, $2\delta=\eta_i\zeta_i$ for all $i=1,2,3$, and we have
	\begin{equation}\label{zeta}
	\zeta_i=-\varphi_i{\mathcal Z}+2\delta\xi_i.
	\end{equation}
Consequently, we can state the following
	
\begin{proposition}\label{proposition-subalgebra}
	Let $({\mathfrak g},\varphi_i,\xi_i,\eta_i)$ be an almost $3$-contact Lie algebra with abelian structure. Then the vertical subspace ${\mathfrak v}$ is a subalgebra of $\mathfrak g$ if and only if ${\mathcal Z}=0$, in which case either $\mathfrak v$ is abelian or isomorphic to $\mathfrak{so}(3)$, according to $\delta=0$ or $\delta\ne 0$ respectively.
\end{proposition}
	
\
	
We prove now a second basic lemma about abelian almost 3-contact structures.
	
\begin{lemma}\label{lemma}
	Let $({\mathfrak g},\varphi_i,\xi_i,\eta_i)$ be an almost $3$-contact Lie algebra with abelian structure. Then, for every  even permutation $(i,j,k)$ of $(1,2,3)$ and for every $X,Y\in\mathfrak h$, the following hold:
		\begin{enumerate}
			\item[\textsc{1.}] $[\xi_i,\varphi_iX]=[\xi_j,\varphi_jX];$
			\item[\textsc{2.}] $[\xi_i,\varphi_jX]=-[\xi_j,\varphi_iX]=[\xi_k,X]=\varphi_i[\xi_j,X]=-\varphi_j[\xi_i,X]$;
			\item[\textsc{3.}] $[\xi_i,X]\in\mathfrak h$;
			\item[\textsc{4.}] $[\varphi_i X,\varphi_jY]=[\varphi_k X,Y]=-[\varphi_j X,\varphi_iY]$.
		\end{enumerate}
\end{lemma}

\begin{proof}
Let us consider $X\in\mathfrak h$. Since $\xi_i,\varphi_iX\in\operatorname{Ker}\eta_k$, we get
		\[[\xi_i,\varphi_iX]=[\varphi_k\xi_i,\varphi_k\varphi_iX]=[\xi_j,\varphi_jX],\]
which gives 1. Analogously,
		\[[\xi_i,\varphi_jX]=[\varphi_k\xi_i,\varphi_k\varphi_jX]=-[\xi_j,\varphi_iX]=-[\varphi_i\xi_j,\varphi_i^2X]=[\xi_k,X].\]
Therefore, we also have $[\xi_i,\varphi_kX]=-[\xi_j,X]$ and  $[\xi_j,\varphi_kX]=[\xi_i,X]$, so that
		\begin{align*}
		[\xi_i,\varphi_jX]&=-[\xi_i,\varphi_i\varphi_kX]=-\varphi_i[\xi_i,\varphi_kX]=\varphi_i[\xi_j,X],\\
		[\xi_j,\varphi_iX]&=[\xi_j,\varphi_j\varphi_kX]=\varphi_j[\xi_j,\varphi_kX]=\varphi_j[\xi_i,X],
		\end{align*}
which completes the proof of 2. As regards 3., taking $X=\varphi_iZ$, $Z\in\mathfrak h$, and applying 1., we have
		\[[\xi_i,X]=[\xi_i,\varphi_iZ]=[\xi_j,\varphi_jZ]=[\xi_k,\varphi_kZ],\]
which belongs to $\mathfrak h$ since each $\mathrm{ad}_{\xi_r}$ commutes with $\varphi_r$. Finally, for every $X,Y\in\mathfrak h$, we have
		\[[\varphi_iX,\varphi_jY]=-[\varphi_j\varphi_iX,Y]=[\varphi_kX,Y]=[\varphi_i\varphi_jX,Y]=-[\varphi_jX,\varphi_iY].\]
\end{proof}
	
\begin{remark}
	With computations analogous to the ones in the proof of Lemma \ref{lemma}, one can prove that if $({\mathfrak g},\varphi_i,\xi_i,\eta_i)$ is an almost $3$-contact Lie algebra such that $\varphi_1$ and $\varphi_2$ are abelian then $\varphi_3$ is also abelian.
\end{remark}	
	
\
	
As a consequence of Lemma \ref{lemma}, given an almost $3$-contact Lie algebra with abelian structure, one can define an endomorphism
\begin{equation}\label{psi}
\psi:{\mathfrak h}\to{\mathfrak h} \qquad \psi:=(\mathrm{ad}_{\xi_i}\circ\varphi_i)|_{\mathfrak h}=(\varphi_i\circ\mathrm{ad}_{\xi_i})|_{\mathfrak h},\quad i=1,2,3.
\end{equation}

\medskip

\begin{proposition}\label{adZ1}
The vector $\mathcal Z\in\h$ satisfies $\psi( \mathcal Z)=0$. In particular, $[\xi,\mathcal Z]=0$ for all $\xi\in\v$.
\end{proposition}

\begin{proof}
	Recall from \eqref{zeta} that $\mathcal Z$ satisfies
	\[ [\xi_j,\xi_k]=-\varphi_i \mathcal{Z} +2\delta \xi_i,\]
	for any even permutation $(i,j,k)$ of $(1,2,3)$. It follows that
	\begin{equation}\label{Znot}
	[\xi_i,[\xi_j,\xi_k]]= -[\xi_i,\varphi_i\mathcal Z]=-\psi (\mathcal Z).
	\end{equation}
	On the other hand, using the Jacobi identity and \eqref{Znot}, we have that
	\[ 0 = [\xi_i,[\xi_j,\xi_k]] + [\xi_j,[\xi_k,\xi_i]] + [\xi_k,[\xi_i,\xi_j]] = -3 \psi (\mathcal Z).   \]
	This implies $\varphi_i([\xi_i,\mathcal Z])=0$ and therefore $[\xi_i,\mathcal Z]=0$ for all $i$.
\end{proof}

\

As an immediate consequence of Proposition \ref{proposition-subalgebra} and Proposition \ref{adZ1} we get

\begin{corollary}\label{coro-invertible}
	If $\psi:\mathfrak h\to \mathfrak h$ is invertible then  $\mathfrak v$ is a subalgebra of $\mathfrak g$.
\end{corollary}

This result will be useful later when we classify 7-dimensional Lie algebras admitting abelian almost 3-contact structures.

\

The next result is an easy consequence of Lemma \ref{lemma}.

\begin{lemma}\label{ker-im}
	$\operatorname{Ker} \ad_{\xi_i}|_{\h}=\operatorname{Ker} \psi$ and $\operatorname{Im} \ad_{\xi_i}|_{\h}=\operatorname{Im}\psi$ for any $i=1,2,3$, and these subspaces of $\h$ are $\varphi_i$-invariant and $\psi$-invariant for any $i=1,2,3$. Therefore they have dimension multiple of $4$.
\end{lemma}

\begin{lemma}\label{Ti}
	Let $({\mathfrak g},\varphi_i,\xi_i,\eta_i)$ be an almost $3$-contact Lie algebra with abelian structure. Then, for any $X\in\h$,
	\begin{enumerate}
		\item[\textsc{1.}] $[\xi_i,[\xi_i,X]]=-\psi^2(X)$, for all $i=1,2,3$,
		\item[\textsc{2.}] $[\xi_i,[\xi_j,X]]=-[\xi_j,[\xi_i,X]]=-\psi([\xi_k,X]) $ for any even permutation $(i,j,k)$ of $(1,2,3)$.
	\end{enumerate}
\end{lemma}

\begin{proof}
Let us denote for simplicity $T_i:=\ad_{\xi_i}|_{\h}:\h \to \h$; also in this proof $\varphi_i$ denotes an endomorphism of $\h$ for all $i$. Note that $T_i$ commutes with $\varphi_i$ and also with $\psi$. For the first statement we compute
\[ \psi^2=T_i\circ \varphi_i \circ T_i \circ \varphi_i=T_i^2\circ \varphi_i^2=-T_i^2.\]
As for the second, we have $T_iT_j=-T_i\circ \varphi_i^2 \circ T_j=-\psi\circ T_k$,
where we have used $\varphi_i\circ T_j=T_k$ from Lemma \ref{lemma}. Analogously, $T_jT_i=\psi\circ T_k$, and the result follows.
\end{proof}

\

We study next the adjoint action of the distinguished element $\mathcal Z\in\h$. We have already proved in Proposition \ref{adZ1} that its action on $\mathfrak v$ is trivial, we analyze next its action on $\mathfrak h$.

We point out first that it follows from the Jacobi identity and Lemma \ref{Ti} that, for $i\neq j$,
\begin{equation}\label{jacobi2}
[[\xi_i,\xi_j],X]=2[\xi_i,[\xi_j,X]], \qquad X\in\mathfrak h.
\end{equation}

Next, for $X\in\mathfrak h$, we compute:
\begin{align*}
[\mathcal Z,X] & = [\varphi_i \mathcal Z,\varphi_i X] \\
& = [2\delta \xi_i-[\xi_j,\xi_k],\varphi_i X] \\
& = 2\delta [\xi_i,\varphi_i X]-2[\xi_j,[\xi_k,\varphi_i X]] \quad \text{(using \eqref{jacobi2})} \\
& = 2\delta \psi( X) -2[\xi_j,[\xi_j, X]] \quad \text{(using Lemma  \ref{lemma})} \\
& = 2(\delta \psi+\psi^2)(X) \quad \text{(using Lemma \ref{Ti})}.
\end{align*}
Therefore,
\begin{equation}\label{adZ}
\operatorname{ad}_{\mathcal Z}|_{\mathfrak h}=2(\psi^2+\delta \psi).
\end{equation}
Equation \eqref{adZ} together with Lemma \ref{adZ1} imply the following:

\begin{corollary}\label{Zcentral}
	$\mathcal Z$ is a central element of $\mathfrak g$ if and only if the endomorphism $\psi$ of $\mathfrak h$ satisfies $\psi^2+\delta \psi=0$.
\end{corollary}

\
	
\subsection{The sphere of abelian almost contact structures}	
	
An almost $3$-contact Lie algebra $(\g,\varphi_i,\xi_i,\eta_i)$ carries a sphere of almost contact structures
\[
\Sigma_{\mathfrak g}=\, \{(\varphi_a,\xi_a,\eta_a) \ |\ a\in S^2 \},
\]
where, for every $a=(a_1,a_2,a_3)\in S^2$,
\begin{equation}\label{structure_sphere}
\varphi_a\, :=\, a_1\varphi_1+a_2\varphi_2+ a_3\varphi_3,\quad
\xi_a \, :=\, a_1\xi_1+a_2\xi_2+a_3\xi_3,\quad \eta_a \, :=\, a_1\eta_1+a_2\eta_2+a_3\eta_3.
\end{equation}
In particular equations \eqref{3-sasaki} are equivalent to
\begin{equation*}
\begin{split}
\varphi_a\circ\varphi_b-\eta_b\otimes\xi_a=\varphi_{a\times b}-(a\cdot b)\,I,\\
\varphi_a\xi_b=\xi_{a\times b}, \quad
\eta_a\circ\varphi_b=\eta_{a\times b},\quad
\end{split}
\end{equation*}
for every $a,b\in S^2$. In general, we shall assume the same definition \eqref{structure_sphere} for every $a\in{\mathbb R}^3$.

\

It is our purpose to prove that when the almost 3-contact structure is abelian then each structure in the associated sphere is abelian. We will use the vectors $\zeta_k=[\xi_i,\xi_j]$ and $\mathcal Z=\varphi_i \zeta_i$  defined previously.

Let us denote for each $a=(a_1,a_2,a_3)\in{\mathbb R}^3$, $\zeta_a:=a_1\zeta_1+a_2\zeta_2+a_3\zeta_3$ so that, for every $a,b\in{\mathbb R}^3$, we have
\begin{equation}\label{general_commutator}
[\xi_a,\xi_b]=\zeta_{a\times b}.
\end{equation}
A direct computation using \textsc{1.} in Lemma \ref{lemma1} and the definition of $\mathcal Z$ shows that for every $a,b\in{\mathbb R}^3$,
\begin{equation}\label{phi_zeta}
\varphi_a\zeta_b=(a\cdot b){\mathcal Z}+\zeta_{a\times b}.
\end{equation}
	
We show next that for every $a,b\in{\mathbb R}^3$,
	\begin{equation}\label{ad_a-phi_b}
	(\mathrm{ad}_{\xi_a}\circ\varphi_b)|_{\mathfrak h}=(a\cdot b)\psi+\mathrm{ad}_{\xi_{a\times b}}|_{\mathfrak h}=(\varphi_a\circ\mathrm{ad}_{\xi_b})|_{\mathfrak h},
	\end{equation}
where $\psi$ is the endomorphism of $\h$ defined in \eqref{psi}. Indeed,
	\begin{align*}
	(\mathrm{ad}_{\xi_a}\circ\varphi_b)|_{\mathfrak h}&=\sum_{i=1}^3a_ib_i (\mathrm{ad}_{\xi_i}\circ\varphi_i)|_{\mathfrak h}+\sum_{1\leq i<j\leq 3}\left(a_ib_j(\mathrm{ad}_{\xi_i}\circ\varphi_j)|_{\mathfrak h}
	+a_jb_i(\mathrm{ad}_{\xi_j}\circ\varphi_i)|_{\mathfrak h}\right)\\
	&=(a\cdot b)\psi+\underset{i,j,k}{\mathfrak S}(a_ib_j-a_jb_i)\,\mathrm{ad}_{\xi_k}|_{\mathfrak h}\\
	&=(a\cdot b)\psi+\mathrm{ad}_{\xi_{a\times b}}|_{\mathfrak h},
	\end{align*}
where $\underset{i,j,k}{\mathfrak S}$ denotes the sum over all even permutations of $(1,2,3)$, and we used the fact that for every even permutation $(i,j,k)$ of $(1,2,3)$,
	\begin{equation}\label{ad-phi1}
	(\mathrm{ad}_{\xi_i}\circ\varphi_j)|_{\mathfrak h}=\mathrm{ad}_{\xi_k}|_{\mathfrak h}=-(\mathrm{ad}_{\xi_j}\circ\varphi_i)|_{\mathfrak h},
	\end{equation}
as proved in 2. of Lemma \ref{lemma}. Analogously, using
	\begin{equation}\label{ad-phi2}
	(\varphi_i\circ\mathrm{ad}_{\xi_j})|_{\mathfrak h}=\mathrm{ad}_{\xi_k}|_{\mathfrak h}=-(\varphi_j\circ\mathrm{ad}_{\xi_i})|_{\mathfrak h},
	\end{equation}
one shows the second equality in \eqref{ad_a-phi_b}.

\
	
\begin{proposition}
	Let $({\mathfrak g},\varphi_i,\xi_i,\eta_i)$ be an almost $3$-contact Lie algebra with abelian structure. Then every structure $(\varphi,\xi,\eta)$ in the sphere $\Sigma_{\mathfrak g}$ is abelian.
\end{proposition}

\begin{proof}
Let $(\varphi_a,\xi_a,\eta_a)$, $a=(a_1,a_2,a_3)\in S^2$, be a structure in the sphere $\Sigma_{\mathfrak g}$ defined as in \eqref{structure_sphere}. Notice that
	\[\operatorname{Ker}\eta_a={\mathfrak h}\oplus{\mathfrak k},\qquad {\mathfrak k}=\{\xi_b\,|\, b\in {\mathbb R}^3,\, a\cdot b=0\}.\]
		
First we show that
		\begin{equation}\label{abelian_sphere_1}
		\mathrm{ad}_{\xi_a}\circ\varphi_a=\varphi_a\circ\mathrm{ad}_{\xi_a}.
		\end{equation}
From \eqref{ad_a-phi_b}, we have
		\[
		(\mathrm{ad}_{\xi_a}\circ\varphi_a)|_{\mathfrak h}=\psi=
		(\varphi_a\circ\mathrm{ad}_{\xi_a})|_{\mathfrak h}.\]
Now, let us consider $\xi_b\in{\mathfrak v}$, with $b\in{\mathbb R}^3$. Applying \eqref{general_commutator} and \eqref{phi_zeta}, we have
		\begin{align*}
		\mathrm{ad}_{\xi_a}(\varphi_a\xi_b)&=[\xi_a,\xi_{a\times b}]=\zeta_{a\times(a\times b)},\\
		\varphi_a(\mathrm{ad}_{\xi_a}\xi_b)&=\varphi_a\zeta_{a\times b}=(a\cdot(a\times b)){\mathcal Z}+\zeta_{a\times(a\times b)}=\zeta_{a\times(a\times b)},
		\end{align*}
thus completing the proof of \eqref{abelian_sphere_1}.
		
We show now that
		\begin{equation}\label{abelian_sphere_2}
		[\varphi_a X,\varphi_a Y]=[X,Y]
		\end{equation}
for every $X,Y\in\operatorname{Ker}\eta_a$. For every $X,Y\in{\mathfrak h}$, we have
		\begin{align*}
		[\varphi_a X,\varphi_a Y]&=\sum_{i=1}^3a_i^2[\varphi_iX,\varphi_iY]+\sum_{1\leq i<j\leq 3}a_ia_j([\varphi_iX,\varphi_jY]+[\varphi_jX,\varphi_iX])\\
		&=(a\cdot a)[X,Y]=[X,Y],
		\end{align*}
where we applied 4. of Lemma \ref{lemma} and the fact that each structure $(\varphi_i,\xi_i,\eta_i)$ is abelian. Now, let us take $X\in{\mathfrak h}$ and $\xi_b\in{\mathfrak k}$, $b\in{\mathbb R}^3$, $a\cdot b=0$. Being $a\in S^2$ and $a\cdot b=0$, we have $(a\times b)\times a=b$. Therefore, applying equation \eqref{ad_a-phi_b}, we get
		\begin{align*}
		[\varphi_a\xi_b,\varphi_a X]&=[\xi_{a\times b}, \varphi_a X]=[\xi_{(a\times b)\times a},X]=[\xi_b,X].
		\end{align*}
		Finally, let us consider $\xi_b,\xi_{b'}\in{\mathfrak k}$, with $b, b'\in{\mathbb R}^3$ such that $a\cdot b=a\cdot b'=0$. Then, being $a\in S^2$,
		\[(a\times b)\times(a\times b')=(a\cdot(b\times b'))a=b\times b',\]
		and therefore,
		\[[\varphi_a\xi_b,\varphi_a\xi_{b'}]=[\xi_{a\times b},\xi_{a\times b'}]=\zeta_{(a\times b)\times(a\times b')}=\zeta_{b\times b'}=[\xi_b,\xi_{b'}],\]
		which completes the proof of \eqref{abelian_sphere_2}.
	\end{proof}
	
\
	
\subsection{Case $\v\subset \z(\g)$}

We analyze next the particular case when the vertical subspace $\mathfrak v$ is contained in the center $\mathfrak z(\mathfrak g)$. In particular, $\mathfrak v$ is an abelian subalgebra of $\mathfrak g$.
	
\begin{lemma}
		If $\xi_i$ is a central element of $\mathfrak g$ for some $i=1,2,3$, then $\xi_i$ is central for all $i=1,2,3$.
\end{lemma}
	
\begin{proof}
	Let $(i,j,k)$ be an even permutation of $(1,2,3)$, with $\xi_i\in \mathfrak z(\mathfrak g)$. Then $\zeta_k=[\xi_i,\xi_j]=0$ and it follows from \eqref{zeta} that $\mathcal Z=0$ and $\delta=0$, so that $\mathfrak v$ is an abelian subalgebra of $\mathfrak g$.
		
	For $X\in\mathfrak h$, it follows from Lemma \ref{lemma} that
	\begin{align*}
	[\xi_k,X] & =[\xi_i,\varphi_j X]=0,\\
	[\xi_j,X] & =-[\xi_i,\varphi_k X]=0.
	\end{align*}
	Thus, both $\xi_j$ and $\xi_k$ are central elements of $\mathfrak z(\mathfrak g)$.
\end{proof}
	
\
	
Assume now that $\mathfrak v \subset \z(\g)$. For $X,Y\in\mathfrak h$ we have a decomposition
\begin{equation}\label{decomposition2}
	[X,Y]=\sigma_1(X,Y)\xi_1+\sigma_2(X,Y)\xi_2+\sigma_3(X,Y)\xi_3+[X,Y]_{\h},
\end{equation}
where $\sigma_i(X,Y)\in\R$ for all $i$ and $[X,Y]_\h\in\h$. Arguing as in the proof of Proposition \ref{proposition-extension} we obtain that $[\cdot,\cdot]_\h$ is a Lie bracket on $\h$ and each $\sigma_i$ is a $2$-cocycle of $\h$ with respect to this Lie algebra structure. If we define
\begin{equation}\label{theta1}
\theta(X,Y):=\sigma_1(X,Y)\xi_1+\sigma_2(X,Y)\xi_2+\sigma_3(X,Y)\xi_3\in\v,
\end{equation}
then clearly $\theta\in \alt^2\h^*\otimes \v$ and, furthermore, $\theta$ is a $\mathfrak v$-valued $2$-cocycle on $\h$. Therefore $\g$ can be considered as a central extension of $\h$ by $\theta$.
	
Next, setting $J_i:=\varphi_i|_{\mathfrak h}$, it follows from \eqref{3-sasaki} that $\{J_1,J_2,J_3\}$ defines an almost hypercomplex structure on $\mathfrak h$. Furthermore, taking into account that $[\varphi_i X,\varphi_i Y]=[X,Y]$ for all $X,Y\in\mathfrak h$ and $i=1,2,3$, and using \eqref{decomposition2}, we arrive at
\[ [J_iX,J_iY]_\h=[X,Y]_\h, \quad \sigma_j(J_iX,J_iY)=\sigma_j(X,Y) \]
for all $i,j=1,2,3$. Thus $\{J_1,J_2,J_3\}$ is in fact an abelian hypercomplex structure on $\h$ and $\sigma_j$ is a $J_i$-invariant
$2$-cocycle of $\h$ for all $i,j$. Equivalently, $\theta$ is a $J_i$-invariant $\v$-valued $2$-cocycle of $\h$ for all $i$.
	
\
	
Summarizing, we have proved the first half of the following result:
	
\begin{proposition}\label{proposition-3extension}
	Let $(\g,\varphi_i,\xi_i,\eta_i)$ be an almost $3$-contact Lie algebra with abelian structure such that $\v \subset \z(\mathfrak g)$. Then the horizontal subspace $\mathfrak h$ admits a Lie algebra structure and an abelian hypercomplex structure $\{J_1,J_2,J_3\}$ such that $\mathfrak g$ is the central extension of the Lie algebra $\mathfrak h$ by a $J_i$-invariant $\mathfrak v$-valued $2$-cocycle $\theta$ of $\mathfrak h$.
		
	Conversely, if $\mathfrak h$ is a Lie algebra equipped with an abelian hypercomplex structure $\{J_1,J_2,J_3\}$ and $\theta$ is a $J_i$-invariant $\R^3$-valued $2$-cocycle on $\h$, then the central extension  $\g=\R^3\oplus_{\theta} \h$ carries a natural abelian almost $3$-contact structure $(\varphi_i,\xi_i,\eta_i)$.
\end{proposition}
	
\begin{proof}
	We only have to prove the second statement. Let $\h$ be a Lie algebra equipped with an abelian hypercomplex structure $\{J_1,J_2,J_3\}$ and let $\theta$ be a $J_i$-invariant $\R^3$-valued 2-cocycle on $\h$, i.e. $\theta\in\alt^2\h^*\otimes \R^3$ satisfies $d\theta=0$ and $\theta(J_iX,J_iY)=\theta(X,Y)$ for any $X,Y\in\h$ and any $i=1,2,3$. More explicitly, there exist a basis $\{\xi_1,\xi_2,\xi_3\}$ of $\R^3$ and $J_i$-invariant 2-cocycles $\sigma_1,\sigma_2,\sigma_3\in\alt^2\h^*$ such that
	\begin{equation}\label{theta}
	\theta(X,Y):=\sigma_1(X,Y)\xi_1+\sigma_2(X,Y)\xi_2+\sigma_3(X,Y)\xi_3\in\R^3 ,
	\end{equation}
	for $X,Y\in\h$. It is clear that the central extension of $\h$ by $\theta$ carries a natural abelian almost 3-contact structure $(\varphi_i,\xi_i,\eta_i)$.
\end{proof}	

\

For instance, we may consider $\sigma_i=df_i$ for some $f_i\in\mathfrak{h}^*$, $i=1,2,3$. Then $\theta$ defined as in \eqref{theta} is $J_i$-invariant and $d\theta=0$.
However, in this case, the corresponding central extension of $\mathfrak h$ by $\theta$ is isomorphic to $\mathbb{R}^3\times \mathfrak h$. Indeed, if $\mathfrak g=\mathbb R^3 \oplus_\theta \mathfrak h$ denotes the central extension of $\mathfrak h$ by $\theta$ then $T:\mathfrak g\to \mathbb R^3 \times \mathfrak h$ defined by
	\[ T(\xi+X)=\xi+f_1(X)\xi_1+f_2(X)\xi_2+f_3(X)\xi_3+X, \]
for $\xi\in\mathbb R^3$ and $X\in\mathfrak h$, is a Lie algebra isomorphism.
In particular, on the subalgebra $\mathfrak{k}=T^{-1}({\mathfrak h})=\{-f_1(X)\xi_1-f_2(X)\xi_2-f_3(X)\xi_3+X\,|\, X\in{\mathfrak h}\}$, one can define the abelian complex structures $J'_i$, $i=1,2,3$, by $J_i'(-\sum_{r}f_r(X)\xi_r+X)=-\sum_rf_r(J_iX)\xi+J_iX$, which satisfies $T(J'_iZ)=J_i(TZ)$ for every $Z\in{\mathfrak k}$.

\
		
\begin{example}\label{affC-3}
 Recall from Proposition \ref{hypercomplex-4d} that $\mathfrak{aff}(\mathbb C)$  is the only non abelian 4-dimensional Lie algebra that admits an abelian hypercomplex structure. Let $\{J_1',J_2',J_3'\}$ be any abelian hypercomplex structure on $\aff(\C)$. According to the discussion following Proposition \ref{hypercomplex-4d}, $\{J_1',J_2',J_3'\}$ gives rise to a sphere which coincides with the sphere determined by the abelian hypercomplex structure $\{J_1,J_2,J_3\}$ given in \eqref{aff-C}.  In particular, a $2$-cocycle $\sigma$ of $\aff(\C)$ is invariant by $\{J_i'\}$ if and only if it is invariant by $\{J_i\}$.
Therefore, using \eqref{d-affC}, a straightforward computation shows that such a $2$-cocycle satisfies $\sigma=x(e^{13}-e^{24})+y(e^{14}+e^{23})$ for some $x,y\in\R$. Hence, $\sigma$ is exact and thus we obtain that any central extension of $\aff(\C)$ by a $J_i'$-invariant $\R^3$-valued  2-cocycle will be isomorphic to $\R^3\times \aff(\mathbb C)$, for any abelian hypercomplex structure $\{J_i'\}$ on $\aff(\C)$.
\end{example}
		
\smallskip
	
We provide now three more examples of almost $3$-contact Lie algebras with abelian structure such that ${\mathfrak v}\subset{\mathfrak z}({\mathfrak g})$. We denote by $\mathfrak g$ a $(4n+3)$-dimensional vector space spanned by vectors $\xi_1,\xi_2,\xi_3, \tau_r,\tau_{n+r},\tau_{2n+r},\tau_{3n+r}$, $r=1,\ldots,n$. Let $\{\eta_i, \theta_l\}$, $i=1,2,3$, $l=1,\dots,4n$, be  the dual basis of  $\{\xi_i, \tau_l\}$, and let $\varphi_i$ be the endomorphism of $\mathfrak g$ defined by
\begin{equation*}
	\varphi_i=\eta_j\otimes\xi_k-\eta_k\otimes\xi_j+\sum_{r=1}^n[\theta_r\otimes\tau_{in+r}
	-\theta_{in+r}\otimes\tau_{r}
	+\theta_{jn+r}\otimes\tau_{kn+r}-\theta_{kn+r}\otimes\tau_{jn+r}]
\end{equation*}
where $(i,j,k)$ is an even permutation of $(1,2,3)$. In particular, one has
\begin{equation}\label{phi-q}
	\varphi_i\tau_r=\tau_{in+r},\quad \varphi_i\tau_{in+r}=-\tau_{r},\quad \varphi_i\tau_{jn+r}=\tau_{kn+r},\quad \varphi_i\tau_{kn+r}=-\tau_{jn+r}.
\end{equation}
In the next examples we consider three different Lie algebra structures on $\g$, so that $(\g,\varphi_i,\xi_i,\eta_i)$ is an almost $3$-contact Lie algebra.
	
\begin{example}\label{ex1}
	Let $\mathfrak g$ be the Lie algebra whose non-vanishing commutators are:
		\begin{align*}
		&[\tau_r,\tau_{2n+r}]=\lambda\xi_1& \quad&[\tau_r,\tau_{n+r}]
		=\lambda\xi_2&\quad &[\tau_r,\tau_{3n+r}]=\lambda\xi_3\\
		&[\tau_{n+r},\tau_{3n+r}]=\lambda\xi_1&\quad &[\tau_{3n+r},\tau_{2n+r}]
		=\lambda\xi_2&\quad &[\tau_{2n+r},\tau_{n+r}]=\lambda\xi_3,
		\end{align*}
	$\lambda\ne 0$ being a real number. This is the Lie algebra $\h^\mathbb{H}_n$ of the quaternionic Heisenberg group $H^\mathbb{H}_n$. By a direct computation one can see that $(\varphi_i,\xi_i,\eta_i)$ is an abelian structure. In fact $\g$ is the central extension of the abelian Lie algebra $\R^{4n}$ endowed with the hypercomplex structure $J_i:=\varphi_i|_\h$, $i=1,2,3$, with $\varphi_i$ defined as in \eqref{phi-q}, by the $J_i$-invariant $\mathfrak v$-valued $2$-cocycle
\begin{align*}
\theta&=\lambda\sum_{r=1}^n\{(\theta_r\wedge\theta_{2n+r}-\theta_{3n+r}\wedge\theta_{n+r})\otimes\xi_1
+(\theta_r\wedge\theta_{n+r}-\theta_{2n+r}\wedge\theta_{3n+r})\otimes\xi_2\\&
\quad+(\theta_r\wedge\theta_{3n+r}-\theta_{n+r}\wedge\theta_{2n+r})\otimes\xi_3\}.
\end{align*}
\end{example}
	
\begin{example}\label{ex2}
	Let $\mathfrak g$ be the Lie algebra with non-vanishing commutators
		\[ [\tau_r,\tau_{2n+r}]=[\tau_{n+r},\tau_{3n+r}]=\lambda\xi_1, \qquad [\tau_r,\tau_{n+r}]=[\tau_{3n+r},\tau_{2n+r}]
		=\lambda\xi_2,\]
	$\lambda\ne 0$ being a real number. This is the Lie algebra of $H_{n}^{\mathbb C}\times \mathbb R$, where $H_{n}^{\mathbb C}$ is the complex Heisenberg group of real dimension $4n+2$, with Lie algebra $\h^\mathbb{C}_{n}$. Then $(\varphi_i,\xi_i,\eta_i)$ is an abelian structure. In particular $\g$ is the central extension of $\R^{4n}$ by the $\mathfrak v$-valued $2$-cocycle
\[ \theta=\lambda\sum_{r=1}^n\{(\theta_r\wedge\theta_{2n+r}-\theta_{3n+r}\wedge\theta_{n+r})\otimes\xi_1
+(\theta_r\wedge\theta_{n+r}-\theta_{2n+r}\wedge\theta_{3n+r})\otimes\xi_2\}.\]
\end{example}
	
\begin{example}\label{ex3}
	Let $\mathfrak g$ be the Lie algebra with non-vanishing commutators
		\[[\tau_r,\tau_{2n+r}]=[\tau_{n+r},\tau_{3n+r}]=\lambda\xi_1\]
	$\lambda\ne 0$ being a real number. This is the Lie algebra of $H_{2n}^{\mathbb R}\times {\mathbb R}^2$, where $H_{2n}^{\mathbb R}$ is the real Heisenberg group of real dimension $4n+1$, with Lie algebra $\h_{2n}^\mathbb{R}$ defined as in Example \ref{Heisenberg}. Then $(\varphi_i,\xi_i,\eta_i)$ is an abelian structure. In this case $\g$ is the central extension of $\R^{4n}$ by the $\mathfrak v$-valued $2$-cocycle
\[ \theta=\lambda\sum_{r=1}^n(\theta_r\wedge\theta_{2n+r}-\theta_{3n+r}\wedge\theta_{n+r})\otimes\xi_1.\]
\end{example}

We will consider again all the structures belonging to Examples \ref{ex1}, \ref{ex2}, \ref{ex3}, as examples of \emph{parallel canonical} structures (see Remark \ref{different}).
	
\

\section{Special cases for the rank of $\psi$ and classification in dimension 7}

Let us recall that the endomorphism $\psi$ and the vector $\mathcal Z$ are related by \eqref{adZ}. In this section we will analyze particular cases, namely, when $\psi$ is invertible and when $\psi=0$. These are the only possibilities in dimension $7$, in which case we obtain a classification of almost $3$-contact Lie algebras with abelian structure.

\

\subsection{Maximum and minimum rank for $\psi$}
Firstly, we consider the case when $\psi$ is an invertible endomorphism of $\h$.

\begin{proposition}\label{psi1}
	Let $({\mathfrak g},\varphi_i,\xi_i,\eta_i)$ be an almost $3$-contact Lie algebra with abelian structure such that the endomorphism $\psi$ is invertible. Then the vertical subspace $\v$ is a subalgebra isomorphic to $\mathfrak{so}(3)$, $\h$ is an abelian ideal, and the non-zero brackets are given by
	\begin{equation}\label{beta-delta}
	[\xi_i,\xi_j]=2\delta \xi_k, \qquad [\xi_i,X]= \delta \varphi_i X,
	\end{equation}
	for any even permutation $(i,j,k)$ of $(1,2,3)$,  $X\in\mathfrak h$ and some $\delta \neq 0$. That is, $\g\cong \mathfrak{so}(3)\ltimes_\rho \R^{4n}$, where $\rho:\mathfrak{so}(3)\to \mathfrak{gl}(4n,\R)$ is a representation of $\mathfrak{so}(3)$ such that $\{ \delta^{-1}\rho(\xi_i):i=1,2,3\}$ is a hypercomplex structure on $\R^{4n}$.
\end{proposition}

\begin{proof}
It follows from Corollary \ref{coro-invertible} that $\v$ is a subalgebra of $\g$, which is equivalent to $\mathcal Z=0$ according to Proposition \ref{proposition-subalgebra}. In particular, we have that $[\xi_i,\xi_j]=2\delta \xi_k$ for some $\delta\in\R$. Since $\mathcal Z=0$, equation \eqref{adZ} becomes $\psi^2+\delta \psi=0$, and from the fact that $\psi$ is invertible we obtain that $\delta\neq 0$ and $\psi=-\delta I|_{\h}$. Hence, $\v\cong\mathfrak{so}(3)$ and, moreover, it follows from the definition of $\psi$ that $[\xi_i,X]=\delta \varphi_i X$ for any $X\in\h$ and $i=1,2,3$.

We compute now, for $X,Y\in\mathfrak h$,
\begin{align*}
[\xi_i, [X,Y]] & = [[\xi_i,X],Y]+[X,[\xi_i,Y]]\\
               & = \delta([\varphi_i X,Y]+[X,\varphi_i Y])\\
               & =0.
\end{align*}
If we decompose $[X,Y]$ as $[X,Y]=[X,Y]_\v+[X,Y]_\h$, with $[X,Y]_\v\in\mathfrak v$ and $[X,Y]_\h\in \mathfrak h$, we obtain that
\begin{align*}
0  & = [\xi_i, [X,Y]]  \\
   & = [\xi_i,[X,Y]_\v] +[\xi_i,[X,Y]_\h]\\
   & = [\xi_i,[X,Y]_\v] + \delta \varphi_i([X,Y]_\h),
\end{align*}
with $[\xi_i,[X,Y]_\v]\in \mathfrak v$ and $\varphi_i([X,Y]_\h)\in\mathfrak h$, since $\mathfrak v$ is a subalgebra and $\varphi_i$ preserves $\mathfrak h$. Therefore
\[ [\xi_i,[X,Y]_\v]=0, \qquad \varphi_i([X,Y]_\h)=0. \]
Since  $\varphi_i|_{\mathfrak h }$ is an isomorphism, we have that $[X,Y]_\h=0$.

It follows from $[\xi_i,[X,Y]_\v]=0$ for all $i$ that $[X,Y]_\v=0$, since $\mathfrak{so}(3)$ has trivial center. Thus $[X,Y]=0$ and $\mathfrak h$ is an abelian subalgebra of $\mathfrak g$. Moreover, since $\operatorname{ad}_{\xi_i}$ preserves $\mathfrak h$ for all $i$, $\mathfrak h$ is an abelian ideal of $\mathfrak g$.
\end{proof}

\

Let us move now to the case when $\psi=0$. We point out that in this case $\mathcal Z$ is a central element of $\g$, according to Corollary \ref{Zcentral}. In the next two propositions we distinguish the cases $\mathcal Z=0$ and $\mathcal Z\neq 0$.

\begin{proposition}\label{psi2}
Let $({\mathfrak g},\varphi_i,\xi_i,\eta_i)$ be an almost $3$-contact Lie algebra with abelian structure such that $\psi=0$ and $\mathcal Z=0$.
\begin{enumerate}
	\item[\textsc{1.}] If $\delta=0$ then $\v$ is in the center of $\g$ and therefore $\g$ is a central extension $\g=\v\oplus_\theta \h$ for certain $J_i$-invariant $\v$-valued $2$-cocycle $\theta$ of $\h$, where $J_i=\varphi_i|_\h$, $i=1,2,3$.
	\item[\textsc{2.}] If $\delta\neq 0$  then both $\v$ and $\h$ are ideals of $\g$ and therefore $\g=\v\times \h$, with $\v\cong \mathfrak{so}(3)$ and $\h$ carries an abelian hypercomplex structure.
\end{enumerate}
\end{proposition}

\begin{proof}
As $\psi=0$ we have that $[\xi,X]=0$ for all $\xi\in\mathfrak v$ and $X\in \mathfrak h$. Also, $\v$ is a subalgebra of $\g$ since $\mathcal Z=0$.

If $\delta=0$ then $\v$ is abelian and therefore is contained in $\z(\g)$. The statement follows from Proposition \ref{proposition-3extension}.

If $\delta\neq 0$ then $\v\cong\mathfrak{so}(3)$, and we will show next that $\h$ is a subalgebra. For $X,Y\in\mathfrak h$ we may decompose $[X,Y]$ as $[X,Y]=[X,Y]_\v+[X,Y]_\h$, with $[X,Y]_\v\in\mathfrak v$ and $[X,Y]_\h\in \mathfrak h$. For $\xi\in \mathfrak v$, we compute
\[ 0=[\xi,[X,Y]]+[X,[Y,\xi]]+[Y,[\xi,X]]=[\xi,[X,Y]_\v+[X,Y]_\h]=[\xi,[X,Y]_\v].\]
Since $\mathfrak v$ has trivial center, we obtain $[X,Y]_\v=0$, thus $\mathfrak h$ is a subalgebra of $\mathfrak g$, equipped with an abelian hypercomplex structure $\{J_i=\varphi_i|_\h\}$. Moreover, since $[\mathfrak v, \mathfrak h]=0$, we obtain that $\mathfrak g \cong \v\times \h$.
\end{proof}

\medskip

\begin{proposition}\label{psi3}
Let $({\g},\varphi_i,\xi_i,\eta_i)$ be an almost $3$-contact Lie algebra with abelian structure such that $\psi=0$ and $\mathcal Z\neq 0$. Then $\h$ is an ideal of $\g$ with an abelian hypercomplex structure. Moreover, the center of $\h$ contains $\mathcal Z$ and $\mathcal Z_i:=\varphi_i \mathcal Z$ for $i=1,2,3$.

The adjoint action of $\xi_i, \, i=1,2,3$, is given by
\begin{equation}\label{adjoint}
 [\xi_i,\xi_j]=2\delta\xi_k-\mathcal Z_k, \qquad [\xi_i,X]=0,
\end{equation}
for any even permutation $(i,j,k)$ of $(1,2,3)$ and for any $X \in\h$. If $\delta\neq 0$ then $\g\cong \u\times \h$, where $\u=\text{span}\{2\delta \xi_1-\mathcal{Z}_1,2\delta \xi_2-\mathcal{Z}_2, 2\delta \xi_3-\mathcal{Z}_3 \}$ is an ideal isomorphic to $\mathfrak{so}(3)$.
\end{proposition}

\begin{proof}
Let us note first that the first equation in \eqref{adjoint} is nothing but \eqref{zeta}, while the second one follows from $\psi=0$.

We prove next that $\h$ is a subalgebra of $\g$. For $X,Y\in\h$, using the Jacobi identity and \eqref{adjoint} we have that
\[ [\xi_i,[X,Y]]=[[\xi_i,X],Y]+[X,[\xi_i,Y]]=0.\]
If we decompose $[X,Y]$ as $[X,Y]=[X,Y]_\v+[X,Y]_\h$, with $[X,Y]_\v\in\mathfrak v$ and $[X,Y]_\h\in \mathfrak h$, then we obtain $[\xi_i,[X,Y]_\v]=0$. Let us further write $[X,Y]_\v=a_1\xi_1+a_2\xi_2+a_3\xi_3$ for some $a_1,a_2,a_3\in\R$. Using \eqref{adjoint} again we get
\[   0= [\xi_1,a_1\xi_1+a_2\xi_2+a_3\xi_3]= a_2(2\delta \xi_3-\mathcal{Z}_3)-a_3(2\delta \xi_2-\mathcal{Z}_2),    \]
and therefore $a_2=a_3=0$. Analogously, considering $[\xi_2,[X,Y]_\v]=0$ we also obtain $a_1=0$, hence $[X,Y]_\v=0$ and $\h$ is a subalgebra. The fact that $\h$ is an ideal follows from $[\v,\h]=0$.

The last statement follows from $[2\delta \xi_i-\mathcal{Z}_i,2\delta \xi_j-\mathcal{Z}_j]=4\delta^2 (2\delta \xi_k-\mathcal{Z}_k)$, for any even permutation $(i,j,k)$ of $(1,2,3)$, which implies that $\u\cong\mathfrak{so}(3)$.
\end{proof}
	
\medskip

\subsection{Classification in dimension $7$}

Here we will provide the classification of $7$-dimensional Lie algebras admitting an abelian almost $3$-contact structure, using the results obtained in the previous subsection.

\begin{theorem}\label{theo7dim}
Let $\g$ be a non-abelian $7$-dimensional Lie algebra admitting an abelian almost $3$-contact structure $(\varphi_i,\xi_i,\eta_i)$. Then $\psi$ is invertible or $\psi=0$. Moreover,
\begin{enumerate}
	\item[\emph{(a)}] If $\psi$ is invertible, then $\g\cong\mathfrak{so}(3)\ltimes \R^4$ with brackets given as in \eqref{beta-delta}.
    \item[\emph{(b)}] If $\psi=0$ and $\mathcal Z=0$, then $\g$ is isomorphic to one of the following Lie algebras: $\R^3\times\aff(\C),\, \h_2^\R\times\R^2,\, \h_1^\C\times\R, \, \h_1^\H, \, \mathfrak{so}(3)\times \R^4, \, \mathfrak{so}(3)\times \aff(\C)$.
    \item[\emph{(c)}] If $\psi=0$ and $\mathcal Z\neq 0$, then $\g$ is isomorphic to one of the following Lie algebras: $\mathfrak{so}(3)\times \R^4$ or $\n\times \R$, where $\n$ is the free $2$-step nilpotent Lie algebra on $3$ generators.
\end{enumerate}
\end{theorem}

\begin{proof}
It follows from Lemma \ref{ker-im} that $\dim\operatorname{Ker}\psi$ is $0$ or $4$, therefore either $\psi$ is invertible or $\psi=0$.

(a) This follows by applying Proposition \ref{psi1}, since $\psi$ is invertible.

\medskip

(b) When $\psi=0$ and $\mathcal Z=0$, we apply Proposition \ref{psi2}. If we assume $\delta=0$ then $\g$ is isomorphic to a central extension $\mathbb R^3\oplus_\theta \mathfrak h$, where $\mathfrak h$ is a 4-dimensional Lie algebra equipped with an abelian hypercomplex structure $\{J_i\}$ and $\theta$ is a $J_i$-invariant $\mathbb R^3$-valued 2-cocycle on $\mathfrak h$. According to Proposition \ref{hypercomplex-4d}, we have that $\mathfrak h$ is isomorphic to $\mathbb R^4$ or $\mathfrak{aff}(\mathbb C)$.

If $\mathfrak h\cong \mathfrak{aff}(\mathbb C)$ then $\mathfrak g\cong \mathbb R^3\times \mathfrak{aff}(\mathbb C)$ (see Example \ref{affC-3}).

If $\mathfrak h\cong \mathbb R^4$ then we may assume that the hypercomplex structure $\{J_i\}$ on $\R^4$ is given by
\begin{equation}\label{hypercomplex}
 J_1=\begin{pmatrix} & -1 && \\ 1 &&& \\ &&& -1\\ && 1 & \end{pmatrix}, \quad J_2=\begin{pmatrix} && -1 & \\  &&& 1 \\ 1 &&& \\ & -1 && \end{pmatrix}, \quad J_3=\begin{pmatrix} &&& -1 \\ && -1 & \\ & 1 && \\ 1 &&& \end{pmatrix},
\end{equation}
in some ordered basis $\{e_1,e_2,e_3,e_4\}$. Accordingly, the Lie bracket on $\mathfrak g$ is given by
\begin{align*}
[e_1,e_2] =  -[e_3,e_4] & =  \alpha_1\xi_1+\beta_1\xi_2+\gamma_1\xi_3 \\
[e_1,e_3] =  [e_2,e_4] & =  \alpha_2\xi_1+\beta_2\xi_2+\gamma_2\xi_3 \\
[e_1,e_4] =  -[e_2,e_3] & =  \alpha_3\xi_1+\beta_3\xi_2+\gamma_3\xi_3,
\end{align*}
for some $\alpha_i,\beta_i,\gamma_i\in \mathbb R$. Therefore, the Lie bracket on $\mathfrak g$ is encoded by the $3\times 3$ matrix
\[ A=\begin{pmatrix} \alpha_1 & \alpha_2 & \alpha_3 & \\ \beta_1 & \beta_2 & \beta_3  \\ \gamma_1 & \gamma_2 & \gamma_3 \end{pmatrix}. \]
Note that $A\neq 0$ since $\g$ is not abelian, in fact, $\mathfrak g$ is $2$-step nilpotent. Moreover, the rank of $A$ determines the isomorphism class of the Lie algebra. Indeed,  if $\text{rank}(A)=1$ then $\mathfrak g\cong \mathfrak h_2^{\mathbb R}\times \mathbb R^2$; if $\text{rank}(A)=2$ then $\mathfrak g\cong \mathfrak h_1^{\mathbb C}\times \mathbb R$, and if $\text{rank}(A)=3$ then $\mathfrak g\cong \mathfrak h_1^{\mathbb H}$.

If we assume $\delta\neq 0$, we obtain that $\g\cong \mathfrak{so}(3)\times \mathfrak h$, where $\mathfrak h$ is an ideal equipped with an abelian hypercomplex structure.  It follows from Proposition \ref{hypercomplex-4d} that either $\mathfrak h\cong \mathbb{R}^4$ or $\mathfrak h\cong \mathfrak{aff}(\mathbb C)$.

\medskip

(c) When $\psi=0$ and $\mathcal Z\neq 0$, we apply Proposition \ref{psi3}. In this case $\h$ is a $4$-dimensional Lie algebra equipped with an abelian hypercomplex structure and non trivial center, therefore $\h=\R^4$.

If $\delta =0$ then the only non-zero brackets are
\[ [\xi_i,\xi_j]=-\mathcal{Z}_k,\]
for any even permutation $(i,j,k)$ of $(1,2,3)$. Clearly, $\mathfrak g$ is a 2-step nilpotent Lie algebra. Moreover, $\mathfrak g=\mathfrak n\times \mathbb R$, where $\mathfrak n=\text{span}\{\xi_1,\xi_2,\xi_3,\mathcal{Z}_1,\mathcal{Z}_2,\mathcal{Z}_3 \}$ is isomorphic to the free 2-step nilpotent Lie algebra on 3 generators.

If $\delta\neq 0$ then the only non-zero brackets are
\[ [\xi_i,\xi_j]=2\delta \xi_k-\mathcal{Z}_k,\]
for any even permutation $(i,j,k)$ of $(1,2,3)$. Moreover, $\mathfrak g \cong \u\times \mathbb R^4$ where $\u=\text{span}\{2\delta \xi_1-\mathcal{Z}_1,2\delta \xi_2-\mathcal{Z}_2, 2\delta \xi_3-\mathcal{Z}_3\}$ is an ideal isomorphic to $\mathfrak{so}(3)$.
\end{proof}

\smallskip

\begin{remark}
Concerning the proof of (2) in the theorem above, in the case $A\in GL(3,\mathbb R)$, that is, $\mathfrak g\cong \mathfrak h_1^{\mathbb H}$, let us consider an inner product on $\mathfrak g$ such that the basis $\{\xi_1,\xi_2,\xi_3,e_1,e_2,e_3,e_4\}$ is orthonormal. Then it can be seen that $\mathfrak g$ with this inner product is an $H$-type Lie algebra if and only if $A\in O(3)$ (see \cite{K} for the relevant definitions).
\end{remark}

\begin{remark}
The Lie algebra $\mathfrak{so}(3)\ltimes \mathbb R^4$ has already appeared in other contexts. For instance, it is proved in \cite{FT} that $SU(2)\ltimes \mathbb R^4$, the associated simply connected Lie group, carries a weakly integrable generalized $G_2$-structure with respect to a non-zero closed $3$-form $H$. Moreover, this group admits compact quotients which inherit this structure (see also Section \ref{compact}). In \cite{CF} it is shown that this Lie group also admits a left invariant contact $SU(3)$-structure.
	
On the other hand, the 2-step nilpotent Lie algebra $\mathfrak n\times \mathbb R$ admits a calibrated $G_2$-structure, according to \cite{CFe}. Since the structure constants are rational, there exists a cocompact discrete subgroup, due to a well known criterion by Malcev (\cite{Ma}); the quotient, called a nilmanifold, carries therefore an induced calibrated $G_2$-structure.
\end{remark}

\

\section{Canonical abelian almost $3$-contact structures}\label{section-canonical}

In this section we will introduce the class of \emph{canonical abelian} almost $3$-contact structures. As motivated in the introduction, referring to the notion of canonical almost $3$-contact metric manifolds introduced in \cite{AgDi}, we will investigate the existence of a canonical connection on a Lie group $G$, endowed with a left invariant abelian almost $3$-contact metric structure $(\varphi_i,\xi_i,\eta_i,g)$.

We recall some basic facts about connections with totally skew-symmetric torsion. For more details we refer to \cite{Ag}.

A metric
connection $\nabla$ with torsion $T$ on a Riemannian manifold $(M,g)$ is said to have \emph{totally skew-symmetric torsion},
or \emph{skew torsion} for short, if the $(0,3)$-tensor field $T$ defined by
\[T(X,Y,Z)=g(T(X,Y),Z)\]
is a $3$-form. The relation between $\nabla$ and the Levi-Civita connection
$\nabla^g$ is then given by
\begin{equation}\label{nabla}
\nabla_XY=\nabla^g_XY+\frac{1}{2}T(X,Y).
\end{equation}
It is well known that $\nabla$ has the same geodesics as $\nabla^g$. The curvature tensor of $\nabla$, defined by $R(X,Y)=[\nabla_X,\nabla_Y]-\nabla_{[X,Y]}$, satisfies
\begin{equation*}
g(R(X,Y)Z,W)=-g(R(X,Y)W,Z),
\end{equation*}
but the Bianchi identities are more complicated than the Riemannian ones. As a consequence, in general the identity
\begin{equation}\label{coppie}
g(R(X,Y)Z,W)=g(R(Z,W)X,Y)\end{equation}
is not satisfied and the Ricci tensor of $R$ is not necessarily symmetric. Nevertheless, if $\nabla$ has parallel torsion, i.e. $\nabla T=0$, then \eqref{coppie} holds and the Ricci tensor is symmetric. Actually, for a metric connection with totally skew-symmetric torsion $T$, the Ricci tensor is symmetric if and only if the torsion is coclosed (\cite{IP}).

\

\subsection{Connections with skew torsion on Hermitian and almost contact metric manifolds}
On any Hermitian manifold $(M,J,g)$ there exists a unique metric connection $\nabla^b$ with totally skew-symmetric torsion such that $\nabla^b J=0$. This connection is known as the Bismut connection and its torsion is the $3$-form given by $c(X,Y,Z)=Jd\Omega(X,Y,Z):=-d\Omega(JX,JY,JZ)$, where $\Omega=g(\cdot,J\cdot)$ is the K\"ahler form.
If $(\h,J,g)$ is a Hermitian Lie algebra with abelian complex structure $J$, considering any Lie group $H$ with Lie algebra $\h$ and the corresponding left invariant structure $(J,g)$, the torsion of the Bismut connection on $H$ is given by
\begin{equation}\label{bismut}
c(X,Y,Z)=-g([X,Y],Z)-g([Y,Z],X)-g([Z,X],Y)
\end{equation}
for every $X,Y\in\h$. In particular, the Bismut connection $\nabla^b$ satisfies
\[g(\nabla^b_XY,Z)=-g(X,[Y,Z])\]
for every $X,Y,Z\in\h$ (\cite{DF}).

\

In \cite{FI} T. Friedrich and S. Ivanov proved that an almost contact metric manifold $(M,\varphi,\xi,\eta,g)$ admits a metric connection with totally skew-symmetric torsion $\nabla$ such that $\nabla\varphi=\nabla\eta=\nabla\xi=0$ if and only if
\begin{enumerate}
\item[(a)] the tensor $N_\varphi$ (defined in \eqref{normality-tensor}) is totally skew-symmetric,
\item[(b)] $\xi$ is a Killing vector field.
\end{enumerate}
The connection $\nabla$ is uniquely determined. It is called the \emph{characteristic connection} of the structure, and its torsion is given by
\begin{equation}\label{torsio-characteristic}
T=\eta\wedge d\eta+N_{\varphi}+d^{\varphi}\Phi-\eta\wedge (\xi\lrcorner\, N_{\varphi}),
\end{equation}
where $d^{\varphi}\Phi$ is defined as
$d^{\varphi}\Phi(X,Y,Z)\, :=\, -d\Phi(\varphi X,\varphi Y,\varphi Z)$. Using this result, we can prove the following

\begin{proposition}\label{characteristic}
Let $\g$ be a Lie algebra endowed with an abelian almost contact metric structure $(\varphi,\xi,\eta,g)$. Let $G$ be any Lie group with Lie algebra $\g$,  with corresponding left invariant structure $(\varphi,\xi,\eta,g)$. Then the following conditions are equivalent:
\begin{enumerate}
\item[\emph{(a)}] $G$ admits a characteristic connection,
\item[\emph{(b)}] $\ad_\xi:\h\to\h$ is skew-symmetric,
\item[\emph{(c)}]$\xi\lrcorner\, d\Phi=0$.
\end{enumerate}
The torsion of the characteristic connection is given by
\[T=\eta\wedge d\eta+c,\]
where $c$ is the $3$-form defined on $\h$ by
\begin{equation}\label{c-characteristic}
c(X,Y,Z)=-g([X,Y],Z)-g([Y,Z],X)-g([Z,X],Y).
\end{equation}
\end{proposition}
\begin{proof}
Since the structure is abelian, it is normal. Therefore, the Lie group $G$ admits a characteristic connection if and only if $\xi$ is Killing, which is equivalent to requiring $\ad_\xi$ to be skew-symmetric on $\g$. Since $\h$ is $\ad_\xi$-invariant, this is equivalent to the skew-symmetry of $\ad_\xi:\h\to\h$. On the other hand, for every $X,Y\in \h$, we have
\begin{align*}
	d\Phi(\xi,X,Y)&=-\Phi([\xi,X],Y)-\Phi([X,Y],\xi)-\Phi([Y,\xi],X)\\
	              &=-g(\mathrm{ad}_\xi X,\varphi Y)+g(\mathrm{ad}_\xi Y,\varphi X)\\
		          &=-g(\mathrm{ad}_\xi X,\varphi Y)-g(\mathrm{ad}_\xi \varphi Y,X),
	\end{align*}
so that $\ad_\xi:\h\to\h$ is skew-symmetric if and only if $\xi\lrcorner\, d\Phi=0$. The torsion of the characteristic connection is given by \eqref{torsio-characteristic}, where $N_\varphi=0$. Furthermore, $\xi\lrcorner\, d^{\varphi}\Phi=0$ and for every $X,Y,Z\in\h$
\begin{align*}
		d^{\varphi}\Phi(X,Y,Z)&={}-d\Phi (\varphi X,\varphi Y,\varphi Z)\\&=\Phi([\varphi X,\varphi Y],\varphi Z)+\Phi ([\varphi Y,\varphi Z],\varphi X)+\Phi ([\varphi Z,\varphi X],\varphi Y)\\
		&=\Phi ([X,Y],\varphi Z)+\Phi ([Y,Z],\varphi X)+\Phi ([Z,X],\varphi Y)\\
		&={}-g([X,Y],Z)-g([Y,Z],X)-g([Z,X],Y).
		\end{align*}
\end{proof}

\begin{remark}
Let us consider an almost contact metric Lie algebra $(\g,\varphi,\xi,\eta,g)$ with abelian structure such that $\xi$ is central. Then any Lie group $G$ with Lie algebra $\g$ admits a characteristic connection. In this case $\g$ is the $1$-dimensional central extension of a Hermitian Lie algebra $(\h,J,g)$ with abelian complex structure,  by a $J$-invariant $2$-cocycle $\sigma$.
The bracket of $\g$ applied to elements in $\h$ can be expressed as
 	\begin{equation*}
		[X,Y]=\sigma(X,Y)\xi+[X,Y]_\h,
	\end{equation*}
where $\sigma(X,Y)\in\R$ and  $[X,Y]_\h\in\h$.
Since $\xi$ and $\h$ are orthogonal we have that the $3$-form $c\in\alt^3\h^\ast$ defined in \eqref{c-characteristic} is given by
\begin{equation*}
	c(X,Y,Z)=-g([X,Y]_\h,Z)-g([Y,Z]_\h,X)-g([Z,X]_\h,Y), \quad X,Y,Z\in\h.
\end{equation*}
Therefore $c$ coincides with the torsion form of the corresponding Bismut connection $\nabla^b$ on $(\h,J,g)$.
\end{remark}

\begin{remark}
Let us consider an almost contact metric Lie algebra $(\g,\varphi,\xi,\eta,g)$ with abelian structure such that $d\eta=0$. By Proposition \ref{prop-extension-D}, $\mathfrak g$ is the $1$-dimensional extension of the Hermitian Lie subalgebra $(\h,J,g)$, with abelian complex structure,  by the derivation $D:=\ad_{\xi}:\h\to\h$, which commutes with $J$. If furthermore $D:\h\to \h $ is skew-symmetric, then $\g$ admits a characteristic connection $\nabla$. The torsion $T$ of $\nabla$ satisfies $\xi\lrcorner \,T=0$ and coincides on $\h$ with the torsion form $c$ of the corresponding Bismut connection. Notice that, if $(\h,J,g)$ is a K\"ahler Lie algebra, then $c=T=0$. In fact in this case, $(\g,\varphi,\xi,\eta,g)$ is a coK\"ahler Lie algebra (see Proposition \ref{prop-alpha-coKahler}), and the characteristic connection coincides with the Levi-Civita connection.
\end{remark}

\

\subsection{HKT manifolds and canonical almost $3$-contact metric manifolds}
Let $(M,J_i,g)$, $i=1,2,3$, be a hyperHermitian manifold, that is a hypercomplex manifold endowed with a Riemannian metric $g$ which is compatible with every complex structure $J_i$, $i=1,2,3$. Then $M$ is said to be a \emph{hyperK\"ahler with torsion} (HKT) manifold if it admits a metric connection with totally skew-symmetric torsion such that $\nabla J_i=0$. In fact $M$ is a HKT manifold if and only if $J_1d\Omega_1=J_2d\Omega_2=J_3d\Omega_3$,
where $\Omega_i$ is the K\"ahler form of the structure $(J_i,g)$. The metric connection with skew torsion parallelizing the three complex structures is uniquely determined: it is the connection with skew torsion
\[c=J_1d\Omega_1=J_2d\Omega_2=J_3d\Omega_3,\]
which coincides with the Bismut connection of each Hermitian structure $(J_i,g)$ (\cite{GP}). A HKT structure is called strong if the torsion $c$ is closed, and it is called weak otherwise. When $c=0$, $\nabla$ is the Levi-Civita connection of $g$ and the structure is hyperK\"ahler.

\

Let $\g$ be a Lie algebra endowed with a hypercomplex structure $J_i$, $i=1,2,3$, and a compatible inner product $g$. Then $(J_i,g)$, $i=1,2,3$, is said to be a HKT structure on $\g$ if any Lie group $G$ with Lie algebra $\g$, endowed with the corresponding left invariant structure $(J_i,g)$ is a HKT manifold (\cite{DF}). It is proved in \cite[Proposition 2.1]{DF} that, when the hypercomplex structure on $\g$ is abelian, then the structure on $G$ is HKT and the torsion of the Bismut connection is given by \eqref{bismut}. Furthermore, if the Lie algebra is non-abelian, the HKT structure on $G$ is weak.

 \

Going now to the context of almost $3$-contact structures, we recall the definition and the characterization of \emph{canonical almost $3$-contact metric manifolds} (see \cite{AgDi} for more details).

	\begin{definition}\label{definition-canonical}\cite{AgDi} An almost $3$-contact metric manifold $(M,\varphi_i,\xi_i,\eta_i,g)$ is called
	\emph{canonical}
	if the following conditions are satisfied:
	\begin{enumerate}
		\item[i)] each $N_{\varphi_i}$ is skew-symmetric on ${\mathcal H}$,
		\item[ii)] each $\xi_i$ is a Killing vector field,
		\item[iii)] for any $X,Y,Z\in\Gamma({\mathcal H})$ and any $i,j=1,2,3$,
		\[N_{\varphi_i}(X,Y,Z)-d\Phi_i(\varphi_i X,\varphi_i Y,\varphi_i Z)=N_{\varphi_j}(X,Y,Z)-d\Phi_j(\varphi_j X,\varphi_j Y,\varphi_j Z),
		\]
		\item[iv)] $M$ admits a Reeb Killing function $\beta\in C^\infty(M)$, that is the tensor fields $A_{ij}$ defined on $\mathcal H$ by
		\begin{equation*}
		A_{ij}(X,Y)\ :=\
		g(({\mathcal L}_{\xi_j}\varphi_i)X,Y)+d\eta_j(X,\varphi_i Y)+d\eta_j(\varphi_i X,Y),
		\end{equation*}
		satisfy
		\begin{equation}\label{Aij2}
		A_{ii}(X,Y)=0,\qquad A_{ij}(X,Y)=-A_{ji}(X,Y)=\beta\Phi_k(X,Y),
		\end{equation}
		for every $X,Y\in\Gamma({\mathcal H})$ and every even permutation $(i,j,k)$ of $(1,2,3)$.
	\end{enumerate}
\end{definition}

	\begin{theorem}\cite{AgDi}\label{theo_canonical}
		An almost $3$-contact metric manifold $(M,\varphi_i,\xi_i,\eta_i, g)$ is canonical, with Reeb Killing function $\beta$, if and only if it
		admits a metric connection $\nabla$ with skew torsion such that
		\begin{equation}\label{canonical}
		\nabla_X\varphi_i\, =\, \beta(\eta_k(X)\varphi _j -\eta_j(X)\varphi _k)
		\end{equation}
		for every vector field $X$ on $M$ and for every even permutation $(i,j,k)$ of $(1,2,3)$.
		If such a connection $\nabla$ exists, it is unique and its torsion is given by
		\begin{align*}
		T(X,Y,Z)&= N_{\varphi_i}(X,Y,Z)-d\Phi_i(\varphi_i X,\varphi_i Y,\varphi_i Z),\\
		T(X,Y,\xi_i)&= d\eta_i(X,Y),\\
		T(X,\xi_i,\xi_j)&=-g([\xi_i,\xi_j],X),\\
		T(\xi_1,\xi_2,\xi_3) &=\ 2(\beta -\delta),
		\end{align*}
		for every $X,Y,Z\in\Gamma({\mathcal H})$, and $i,j=1,2,3$. Here $\delta$ is the Reeb commutator function, that is a differentiable function such that $\eta_k([\xi_i,\xi_j]) = 2\delta\epsilon_{ijk}$.
	\end{theorem}
	
The connection $\nabla$ defined in the Theorem above is called the \emph{canonical connection} of $M$.
Besides equation \eqref{canonical}, it satisfies
\begin{equation}\label{derivatives}
\nabla_X\xi_i\, =\, \beta(\eta_k(X)\xi_j -\eta_j(X)\xi_k), \qquad \nabla_X\eta_i\, =\, \beta(\eta_k(X)\eta_j -\eta_j(X)\eta_k)
\end{equation}
for every vector field $X$ on $M$.
	A canonical almost $3$-contact metric manifold is called \emph{parallel}
	if $\beta=0$, in which case the canonical connection parallelizes all the structure tensor fields.

In a canonical almost $3$-contact metric manifold $(M,\varphi_i,\xi_i,\eta_i,g)$, each structure $(\varphi_i,\xi_i,\eta_i,g)$ admits a characteristic connection $\nabla^i$. In general the three characteristic connections do not coincide. In the parallel case, it can be shown that $\nabla^1=\nabla^2=\nabla^3=\nabla$.

\
	
A large class of canonical almost $3$-contact metric manifolds is given by \emph{$3$-$(\alpha,\delta)$-Sasaki} manifolds. These are defined as almost $3$-contact metric manifolds with structure $(\varphi_i,\xi_i,\eta_i,g)$ such that
\begin{equation}\label{differential_eta}
d\eta_i=2\alpha\Phi_i+2(\alpha-\delta)\eta_j\wedge\eta_k
\end{equation}
for every even permutation $(i,j,k)$ of $(1,2,3)$, where $\alpha\ne0$ and
$\delta$ are real constants.
A $3$-$(\alpha,\delta)$-Sasaki manifold is called \emph{degenerate} if
$\delta=0$ and \emph{nondegenerate} otherwise. Quaternionic Heisenberg groups are
examples of degenerate  $3$-$(\alpha,\delta)$-Sasaki manifolds \cite[Example 2.3.2]{AgDi}.
It is known that every $3$-$(\alpha,\delta)$-Sasaki manifold is hypernormal. In particular, when $\alpha=\delta=1$, the manifold is  \emph{$3$-Sasakian}. Furthermore, every $3$-$(\alpha,\delta)$-Sasaki manifold is canonical with constant Reeb Killing function $\beta=2(\delta-2\alpha)$, and the canonical connection has parallel torsion.

\

A second class of canonical almost $3$-contact metric structures is given by \emph{$3$-$\delta$-cosymplectic} structures. They are defined by the conditions
\begin{equation*}
d\eta_i=-2\delta\eta_j\wedge\eta_k,\qquad d\Phi_i=0,
\end{equation*}
for some $\delta\in\R$ and for every even permutation $(i,j,k)$ of $(1,2,3)$.
When $\delta=0$, this is the notion of \emph{$3$-cosymplectic} structure. Every $3$-$\delta$-cosymplectic manifold is hypernormal and locally isometric
to the Riemannian product of a hyperK\"ahler manifold, tangent to the horizontal distribution, and the $3$-dimensional
sphere of constant curvature $\delta^2$, tangent to the vertical distribution. Further, every $3$-$\delta$-cosymplectic manifold is parallel canonical, and the torsion of the canonical connection is given by $T=-2\delta\eta_1\wedge\eta_2\wedge\eta_3$.

\

\subsection{Canonical abelian almost $3$-contact metric structures}

\begin{definition}
	An almost $3$-contact metric structure $(\varphi_i,\xi_i,\eta_i, g)$ on a Lie algebra $\mathfrak g$ is called \emph{(parallel) canonical} if any Lie group $G$ with Lie algebra $\mathfrak g$, endowed with the corresponding left invariant almost $3$-contact metric structure $(\varphi_i,\xi_i,\eta_i, g)$ is (parallel) canonical.
\end{definition}
	
\begin{theorem}\label{proposition-canonical}
	Let $({\mathfrak g},\varphi_i,\xi_i,\eta_i, g)$ be an almost $3$-contact metric Lie algebra with abelian structure. Then the structure is canonical if and only if
		\begin{itemize}
			\item[\textsc{1.}] the vertical subspace $\mathfrak v$ is a subalgebra of $\mathfrak g$, so that $[\xi_i,\xi_j]=2\delta\xi_k$, $\delta\in{\mathbb R}$;
			\item[\textsc{2.}] $\ad_{\xi_i}|_{\mathfrak h}=\frac{\beta}{2}\varphi_i|_{\mathfrak h}$ for every $i=1,2,3$ and for some $\beta\in\mathbb{R}$.
		\end{itemize}
	If the structure is canonical, the canonical connection of a Lie group $G$ with Lie algebra $\g$, endowed with the corresponding left invariant structure, has torsion
\begin{equation}\label{torsion}
T=c+\sum_{i=1}^3\eta_i\wedge d\eta_i+2(\beta +2\delta)\ \eta_1\wedge\eta_2\wedge\eta_3,
\end{equation}
where $c$ is the $3$-form defined on $\h$ by
\begin{equation}\label{c}
 c(X,Y,Z)=-g([X,Y],Z)-g([Y,Z],X)-g([Z,X],Y).
\end{equation}
 	\end{theorem}

\begin{proof}
	Let us consider a Lie group $G$ with Lie algebra $\g$, and the corresponding left invariant structure $(\varphi_i,\xi_i,\eta_i,g)$.	Since the almost $3$-contact structure is abelian, it is hypernormal, so that $N_{\varphi_i}=0$. Furthermore, a computation analogous to one in Proposition \ref{characteristic} shows that for every $X,Y,Z\in\mathfrak h$ and for every $i=1,2,3$, 
	\[d\Phi_i(\varphi_i X,\varphi_i Y,\varphi_i Z)=g([X,Y],Z)+g([Y,Z],X)+g([Z,X],Y),\] 
	so that condition iii) in Definition \ref{definition-canonical} is satisfied.

	Now, we consider the tensors $A_{ij}$ defined by
	\[ A_{ij}(X,Y)\ :=\g(({\mathcal L}_{\xi_j}\varphi_i)X,Y)+d\eta_j(X,\varphi_i Y)+d\eta_j(\varphi_i X,Y), \]
	for every $X,Y\in{\mathfrak h}$. Since the structure is abelian, we have
	\[d\eta_j(X,\varphi_i Y)+d\eta_j(\varphi_i X,Y)=-\eta_j([X,\varphi_i Y]+[\varphi_i X,Y])=0,\]
	so that $A_{ij}(X,Y)=g(({\mathcal L}_{\xi_j}\varphi_i)X,Y)$. From $\mathrm{ad}_{\xi_i}\circ \varphi_i=\varphi_i\circ\mathrm{ad}_{\xi_i}$, and equations \eqref{ad-phi1}, \eqref{ad-phi2}, we have
	\[A_{ii}(X,Y)=0, \qquad A_{ij}(X,Y)=-A_{ji}(X,Y)=-2g(\mathrm{ad}_{\xi_k}X,Y).\]
	Therefore, since both the endomorphisms $\mathrm{ad}_{\xi_k}$ and $\varphi_k$ preserve the subspace $\mathfrak h$, equation \eqref{Aij2} holds if and only if $2\,\mathrm{ad}_{\xi_k}|_{\mathfrak h}=\beta\varphi_k|_{\mathfrak h}$ for some $\beta\in\mathbb{R}$. In particular, this implies that each endomorphism $\mathrm{ad}_{\xi_i}$ is skew-symmetric on $\mathfrak h$.
		
	The last requirement for the structure to be canonical is that each $\xi_i$ should be Killing, which is equivalent to require that $\mathrm{ad}_{\xi_i}$ is skew-symmetric on $\mathfrak g$. Therefore, for every $X\in{\mathfrak h}$ and $r=1,2,3$, we have
	\begin{align*}
		g([\xi_i,X],\xi_r)+g(X,[\xi_i,\xi_r])=0.
	\end{align*}
	Since $[\xi_i,X]\in{\mathfrak h}$, we deduce that the vertical subspace has to be a subalgebra of ${\mathfrak g}$, which according to Proposition \ref{proposition-subalgebra} is isomorphic to either $\mathbb R^3$ or $\mathfrak{so}(3)$. In particular $[\xi_i,\xi_j]=2\delta\xi_k$ for some $\delta\in\mathbb{R}$, and $\mathrm{ad}_{\xi_i}$ is also skew-symmetric on $\mathfrak v$.
		
	Finally, if the abelian structure $(\varphi_i,\xi_i,\eta_i, g)$ on ${\mathfrak g}$ is canonical, the Lie group $G$ admits a unique metric connection $\nabla$ with skew torsion satisfying
	\[\nabla_X\varphi_i\, =\, \beta(\eta_k(X)\varphi _j -\eta_j(X)\varphi _k) \quad		\forall X\in{\mathfrak g}.\]
    From Theorem \ref{theo_canonical}, the torsion of the canonical connection is given by	
\begin{align*}
		T(X,Y,Z)&=-g([X,Y],Z)-g([Y,Z],X)-g([Z,X],Y),\\
		T(X,Y,\xi_i)&= d\eta_i(X,Y),\\
		T(X,\xi_i,\xi_j)&=0,\\
		T(\xi_1,\xi_2,\xi_3) &=\ 2(\beta -\delta),
		\end{align*}
	for every $X,Y,Z\in{\mathfrak h}$, $i,j=1,2,3$. The equations above are equivalent to \eqref{torsion}.
\end{proof}

	\

\begin{remark}
Conditions 1. and 2. in Theorem \ref{proposition-canonical} can be rephrased as
\begin{equation}\label{canonic}
{\mathcal Z}=0,\qquad 2\psi=-\beta I,
\end{equation}
for some real number $\beta$, which implies that $\psi=0$ or $\psi$ is invertible.
\end{remark}	

One can notice that conditions \eqref{canonic} are independent on the metric. Therefore, one can also give the following

\begin{definition}
An almost $3$-contact Lie algebra $(\g, \varphi_i,\xi_i,\eta_i)$ with abelian structure is called \emph{canonical} if ${\mathcal Z}=0$ and $2\psi=-\beta I$, for some $\beta\in\R$. It is called \emph{parallel canonical} if ${\mathcal Z}=0$ and $\psi=0$.
\end{definition}

In particular, every abelian almost $3$-contact structure $(\varphi_i,\xi_i,\eta_i)$ such that $\v\subset\z(\g)$, is parallel
canonical. Therefore, all the structures belonging to Examples \ref{ex1}, \ref{ex2}, \ref{ex3}, are parallel canonical.

\begin{remark}\label{3ad-abelian}
Recall that every $3$-$(\alpha,\delta)$-Sasaki structure is hypernormal and canonical. Nevertheless, notice that every abelian almost $3$-contact metric structure $(\varphi_i,\xi_i,\eta_i,g)$ on a Lie algebra $\g$ cannot be $3$-$(\alpha,\delta)$-Sasaki. Indeed, if the structure is abelian, taking $i,j=1,2,3$, $i\ne j$, for every horizontal vectors $X,Y\in\h$, we have
\[d\eta_i(\varphi_jX,\varphi_jY)=-\eta_i([\varphi_jX,\varphi_jY])=-\eta_i([X,Y])=d\eta_i(X,Y).\]
On the other hand,
\[\Phi_i(\varphi_jX,\varphi_jY)=g(\varphi_jX,\varphi_i\varphi_jY)=-g(\varphi_jX,\varphi_j\varphi_iY)=-\Phi_i(X,Y),\]
so that \eqref{differential_eta} cannot hold.
\end{remark}

\	

Due to \eqref{canonic}, the description of almost $3$-contact Lie algebras with canonical abelian structure can be obtained by Propositions
\ref{psi1} and \ref{psi2}. In the following, taking a corresponding Lie group $G$ endowed with a compatible left invariant Riemannian metric, we separately examine the parallel and non-parallel cases, determining for each case the canonical connection of the structure. We shall also discuss the parallelism of the torsion.

Analyzing parallel canonical abelian structures ($\mathcal Z=0$, $\psi=0$), we distinguish the cases $\delta=0$ and $\delta\ne 0$.

\

\subsubsection{Parallel canonical abelian structures with $\delta=0$}
\begin{proposition}\label{central-extension}
Let $(\mathfrak g, \varphi_i, \xi_i, \eta_i,g)$ be an almost $3$-contact metric Lie algebra, with a parallel canonical abelian structure such that $\delta=0$. Then $\mathfrak v\subset \mathfrak z(\mathfrak g)$ and $\g$ is the central extension $\g=\v\oplus_\theta \h$ of a Lie algebra $\h$ carrying an abelian hypercomplex structure $J_i$, $i=1,2,3$, for the $J_i$-invariant $\v$-valued $2$-cocycle $\theta$ on $\h$ given by
\[\theta(X,Y)=-\sum_{i=1}^3d\eta_i(X,Y)\xi_i.\]
If $G$ is any Lie group with Lie algebra $\g$, endowed with left invariant structure
$(\varphi_i,\xi_i,\eta_i,g)$, the canonical connection of $G$
is the connection with skew torsion
\begin{equation}\label{T1}
T=c+\sum_{i=1}^3\eta_i\wedge d\eta_i.
\end{equation}
	\end{proposition}

\begin{proof}
The first part is a consequence of Proposition \ref{psi2}. We only need to remark that, taking into account \eqref{decomposition2}, \eqref{theta1}, and the inner product $g$, for every $X,Y\in\h$, we have
\[g(\theta(X,Y),\xi_i)=g([X,Y],\xi_i)=\eta_i([X,Y])=-d\eta_i(X,Y).\]
Equation \eqref{T1} immediately follows from \eqref{torsion}.
\end{proof}

\

We will compute explicitly the canonical connection $\nabla$ on a central extension as in Proposition \ref{central-extension}. Keeping the notation in the proposition, the bracket of $\g$ applied to elements in $\h$ can be expressed as
\begin{equation}\label{extension}
  [X,Y]=\theta(X,Y)+[X,Y]_\h\in\v\oplus \h,
  \end{equation}
and since $\v$ and $\h$ are orthogonal we have that the $3$-form $c\in\alt^3\h^\ast$ defined in \eqref{c} is given by
\begin{equation*}
	c(X,Y,Z)=-g([X,Y]_\h,Z)-g([Y,Z]_\h,X)-g([Z,X]_\h,Y), \quad X,Y,Z\in\h.
\end{equation*}
Hence, $c$ is the torsion form of the corresponding Bismut connection $\nabla^b$ on $(\h,[\cdot,\cdot]_\h)$.

Recall that the canonical connection preserves the distributions defined by $\v$ and $\h$. In particular, in this parallel case we have $\nabla\xi=0$ for $\xi\in\v$.
Therefore, we only need to compute $g(\nabla_XY,Z)$ and $g(\nabla_\xi X,Y)$
for $X,Y,Z\in\h$ and $\xi\in\v$. Using \eqref{nabla} and \eqref{T1},  direct computations show that
\begin{align*}
	g(\nabla_XY,Z) & = g(\nabla^b_XY,Z) = -g(X,[Y,Z]),\\
	g(\nabla_\xi X,Y)&  = -g([X,Y],\xi) = -g(\theta(X,Y),\xi).
\end{align*}
To sum up, we have

\begin{proposition}\label{idem-bismut}
The canonical connection $\nabla$ on $\g=\v\oplus_\theta \h$ is given by
\begin{equation}\label{like-Bismut}
g(\nabla_XY,Z)=-g(X,[Y,Z]) \mbox{ for any }X,Y,Z\in\g.
\end{equation}
\end{proposition}

This is the same expression of the Bismut connection on a Lie algebra with an abelian hypercomplex structure and hyperhermitian metric (see \cite{DF}).

\begin{remark}
Notice that, being $\nabla{\eta_i}=0$, if $\eta$ is a $1$-form $g$-dual to an element of $\v$, then $\nabla\eta=0$. Further, $\nabla_A \omega \in \h^\ast$, for all $A\in\g$ and $\omega\in\g^\ast$.
\end{remark}

\

We discuss now the parallelism of the torsion.

\

Recall that $\alt^3\g^\ast=\alt^3(\v\oplus \h)^\ast =\alt^3\v^\ast \oplus \alt^2\v^\ast \otimes \h^\ast \oplus \v^\ast\otimes \alt^2\h^\ast \oplus \alt^3\h^\ast$.
The torsion $3$-form $T$ on $\g$ is given by
\[   T=c+\sum_{i=1}^3\eta_i\wedge d\eta_i,   \]
with $c\in\alt^3\h^\ast$ and $\eta_i\wedge d\eta_i\in \v^\ast\otimes \alt^2\h^\ast$ for all $i$, since $d\eta_i(\xi,\cdot)=0$ for any $\xi\in\v$. Therefore, for any $A\in\g$ we have
\[   \nabla_A c\in \alt^3\h^\ast, \quad \nabla_A (\eta_i\wedge d\eta_i)=\eta_i\wedge \nabla_A d\eta_i\in \v^\ast\otimes \alt^2\h^\ast.  \]
Thus, $T$ is parallel if and only if $c$ and $\eta_i\wedge d\eta_i$ are parallel for all $i$, which is equivalent to $\nabla_A d\eta_i=0$, for all $i=1,2,3$. As a consequence, we get

\begin{proposition}
	The torsion $3$-form $T$ is parallel if and only if $\nabla_A c=0$ and $\nabla_A d\eta_i=0$ for all $A\in\g$ and for all $i=1,2,3$.
\end{proposition}

In particular, if $T$ is parallel then the Bismut connection on $\h$ has parallel torsion.

\

\begin{example}\label{different}
 Consider the Lie algebras described in Examples \ref{ex1}, \ref{ex2}, \ref{ex3}, all admitting a parallel canonical abelian structure with $\delta=0$, obtained as central extensions of the abelian Lie algebra $\R^{4n}$. Take the inner product $g_\lambda$ on $\mathfrak g$ with respect to which the basis of vectors $\xi_i,\tau_l$ is orthonormal. Then $({\mathfrak g},\varphi_i,\xi_i,\eta_i,g_\lambda)$ is an almost $3$-contact metric Lie algebra with abelian structure. If $G$ is any Lie group with Lie algebra $\mathfrak g$, endowed with the left invariant structure
	$(\varphi_i,\xi_i,\eta_i,g_\lambda)$, being $c=0$, the torsion of the canonical connection of $G$ is
\[T=\sum_{i=1}^3\eta_i\wedge d\eta_i.\]
One can notice that the Lie algebra described in Example \ref{ex1} is isomorphic to the Lie algebra described in \cite[Example 2.3.2]{AgDi} via the isomorphism interchanging $\tau_{n+r}$ and $\tau_{2n+r}$. Therefore, two different left invariant almost $3$-contact metric structures are provided on the quaternionic Heisenberg group. In fact, the structure in \cite[Example 2.3.2]{AgDi} is a degenerate $3$-$(\alpha,\delta)$-Sasaki structure, while the structure belonging to Example \ref{ex1} is not $3$-$(\alpha,\delta)$-Sasaki, since it is abelian (see Remark \ref{3ad-abelian}).

Referring to these three examples, endowed with the compatible metric $g_\lambda$, we show that Example \ref{ex3} is the only one for which the canonical connection has parallel torsion. Indeed, in these cases, we have that $\nabla T=0$ if and only if $\nabla d\eta_i=0$. In all the three examples,
\[d\eta_1=-\lambda\sum_{r=1}^n(\theta_r\wedge\theta_{2n+r}-\theta_{3n+r}\wedge\theta_{n+r}).\]
As regards Examples \ref{ex1} and \ref{ex2}, using \eqref{like-Bismut}, one can easily check that
\[\nabla_{\xi_2}\tau_r=-\lambda\tau_{n+r},\qquad \nabla_{\xi_2}\tau_{n+r}=\lambda\tau_{r},\qquad \nabla_{\xi_2}\tau_{2n+r}=\lambda\tau_{3n+r},\qquad \nabla_{\xi_2}\tau_{3n+r}=-\lambda\tau_{2n+r},\]
and thus on the dual $1$-forms, we have
\[\nabla_{\xi_2}\theta_r=-\lambda\theta_{n+r},\qquad \nabla_{\xi_2}\theta_{n+r}=\lambda\theta_{r},\qquad \nabla_{\xi_2}\theta_{2n+r}=\lambda\theta_{3n+r},\qquad \nabla_{\xi_2}\theta_{3n+r}=-\lambda\theta_{2n+r},\]
so that we get
\[\nabla_{\xi_2}d\eta_1=-2\lambda^2\sum_{r=1}^n(\theta_r\wedge\theta_{3n+r}-\theta_{n+r}\wedge\theta_{2n+r})\ne0.\]
As regards Example \ref{ex3}, the only non-vanishing covariant derivatives of the canonical connection are
\[\nabla_{\xi_1}\tau_r=-\lambda\tau_{2n+r},\qquad \nabla_{\xi_1}\tau_{2n+r}=\lambda\tau_{r},\qquad \nabla_{\xi_1}\tau_{n+r}=-\lambda\tau_{3n+r},\qquad \nabla_{\xi_1}\tau_{3n+r}=\lambda\tau_{n+r},\]
and we have analogous equations on the dual $1$-forms, which imply that $\nabla_{\xi_1}d\eta_1=0$. Together with the fact that $d\eta_2=d\eta_3=0$, this implies that $\nabla T=0$.
\end{example}

\medskip

Recall that for a metric connection $\nabla$ with skew torsion $T$ the parallelism of the torsion is a sufficient condition for the Ricci tensor to be symmetric. Nevertheless, we shall see that in all the three cases discussed in Example \ref{different}, the Ricci tensor is symmetric, which is in fact equivalent to the torsion form being coclosed.
\begin{proposition}
If the Lie algebra $\h$ is abelian, the Ricci tensor of the canonical connection on $\g=\v\oplus_\theta \h$ is symmetric.
\end{proposition}
\begin{proof}
For every $A,B\in\g$, the Ricci tensor of the canonical connection is given by
 \[Ric(A,B)=\sum_{i=1}^3g(R(\xi_i,A)B,\xi_i)+\sum_{i=1}^{4n}g(R(e_i,A)B,e_i),\]
 where $\{e_i\}$ is an orthonormal basis of $\h$. Since $\nabla\xi_i=0$, we have
 \[g(R(\xi_i,A)B,\xi_i)=-g(R(\xi_i,A)\xi_i,B)=0.\]
 On the other hand, since $\h$ is abelian, from \eqref{like-Bismut} and \eqref{extension}, we have $\nabla_{e_i}=0$ and thus
  \[Ric(A,B)=-\sum_{i=1}^{4n}g(\nabla_{[e_i,A]}B,e_i)=\sum_{i=1}^{4n}g([e_i,A],[B,e_i]),\]
which implies that $Ric(A,B)=Ric(B,A)$.
\end{proof}

\medskip

\begin{remark}
	The Lie groups in Example \ref{different} are all nilpotent and have rational structure constants, therefore due to Malcev's criterion (\cite{Ma}), they admit co-compact discrete subgroups. The corresponding nilmanifolds admit an induced parallel canonical almost 3-contact metric structure.
\end{remark}

\
\subsubsection{Parallel canonical abelian structures with $\delta\ne0$}
\begin{proposition}\label{parallel-canonical2}
Let $(\mathfrak g, \varphi_i, \xi_i, \eta_i,g)$ be an almost $3$-contact metric Lie algebra, with a parallel canonical abelian
structure such that $\delta\ne0$. Then $\g=\v\times \h$, with $\v\cong \mathfrak{so}(3)$ and $\h$ carrying an abelian hypercomplex structure $J_i$, $i=1,2,3$.
If $G$ is any Lie group with Lie algebra $\mathfrak g$, endowed with left invariant structure
$(\varphi_i,\xi_i,\eta_i,g)$, the canonical connection of $G$
is the connection with skew torsion
\begin{equation}\label{T2}
T=c-2\delta\eta_1\wedge\eta_2\wedge\eta_3.
\end{equation}
	\end{proposition}
\begin{proof}
This is again an immediate consequence of Proposition \ref{psi2} and Theorem \ref{proposition-canonical}. Equation \eqref{T2} follows from \eqref{torsion}, taking into account that in this case $d\eta_i=-2\delta\eta_j\wedge\eta_k$ for every even permutation $(i,j,k)$ of $(1,2,3)$.
\end{proof}

In the assumptions of the Proposition above, the Lie subalgebra $\h$ of $\g$ is endowed with the HKT structure $(J_i,g)$, $i=1,2,3$. The $3$-form $c$ coincides with the torsion of the Bismut connection $\nabla^b$ of any Lie group $H$ with Lie algebra $\h$, and left invariant structure $(J_i,g)$, $i=1,2,3$.

Considering the canonical connection $\nabla$ and applying \eqref{nabla} and \eqref{T2}, one immediately verifies that
\[	g(\nabla_XY,Z)=g(\nabla^b_XY,Z)=-g(X,[Y,Z]),\qquad g(\nabla_\xi X,Y)=0\]
for $X,Y,Z\in\h$ and $\xi\in\v$. Being also $\nabla\xi=0$, we have that $\nabla_XY=\nabla^b_XY\in\h$ for $X,Y\in\h$ and all the other covariant derivatives vanish.

Therefore, taking into account \eqref{T2}, since $\nabla\eta_i=0$, we have that

\begin{proposition}
	The torsion $3$-form $T$ on $\g=\mathfrak{so}(3)\times \h$ is parallel if and only if $\nabla_X c=0$ for all $X\in\h$, or equivalently $\nabla^bc=0$, that is the Bismut connection on $\h$ has parallel torsion.
\end{proposition}

\begin{remark} In the assumptions of Proposition \ref{parallel-canonical2}, if the Lie algebra $\mathfrak h$ is abelian, then $c=0$ and the hyperhermitian structure $(J_i,g)$ is hyperK\"ahler. In this case, if $G$ is the simply connected Lie group with Lie algebra $\g$, then it is isomorphic to $\operatorname{SU}(2)\times \mathbb{R}^{4n}$, and the structure $(\varphi_i,\xi_i,\eta_i,g)$ on $G$ is $3$-$\delta$-cosymplectic.
\end{remark}

\
	
\subsubsection{Non-parallel canonical abelian structures}
\begin{proposition}\label{semidirect}
Let $(\g,\varphi_i,\xi_i,\eta_i,g)$ be an almost $3$-contact metric Lie algebra endowed with a non-parallel canonical abelian structure, i.e. ${\mathcal
Z}=0$ and
$2\psi=-\beta I$, with $\beta\ne 0$. Then $\beta=2\delta$ and $\g\cong\mathfrak{so}(3)\ltimes \R^{4n}$, with non zero brackets given by
	\begin{equation}\label{bracket}
	[\xi_i,\xi_j]=2\delta \xi_k, \qquad [\xi_i,X]= \delta \varphi_i X,
	\end{equation}
	for any even permutation $(i,j,k)$ of $(1,2,3)$. If $G$ is any Lie group with Lie algebra $\g$, endowed with the left invariant structure
$(\varphi_i,\xi_i,\eta_i,g)$, the canonical connection of $G$ is the connection with skew torsion
\begin{equation}\label{T3}
T=2\delta \eta_1\wedge\eta_2\wedge\eta_3.
\end{equation}
\end{proposition}

\begin{proof}
This is a consequence of Proposition \ref{psi1} and Theorem \ref{proposition-canonical}. In particular, even in this case $d\eta_i=-2\delta\eta_j\wedge\eta_k$, for every even permutation $(i,j,k)$ of $(1,2,3)$, and \eqref{T3} follows from \eqref{torsion}.
\end{proof}

\medskip

In the case of Proposition \ref{semidirect}, the canonical connection $\nabla$ has parallel torsion. Indeed, by the second equation in \eqref{derivatives},
\[\nabla_X\eta_i=0,\qquad \nabla_{\xi_j}\eta_i=-2\delta\eta_k,\qquad \nabla_{\xi_k}\eta_i=2\delta\eta_j\]
for every $X\in \h$ and for every even permutation $(i,j,k)$ of $(1,2,3)$. Therefore, taking into account \eqref{T3}, we have $\nabla T=0$.

Further, under the assumption of the Proposition above, using \eqref{bracket}, one can easily check that the fundamental $2$-forms of the structure satisfy
\[d\Phi_i=-2\delta(\eta_j\wedge\Phi_k-\eta_k\wedge\Phi_j)\]
for every even permutation $(i,j,k)$ of $(1,2,3)$, which implies that in this case the structure is not $3$-$\delta$-cosymplectic.
	
\	
	
\section{The Lie algebra $\mathfrak{so}(3)\ltimes \R^{4n}$}\label{compact}

In this section we determine the simply connected Lie group associated to the Lie algebra $\mathfrak{so}(3)\ltimes \R^{4n}$ which, according to Proposition \ref{semidirect}, carries a non-parallel canonical abelian structure. Later we show that for any $n$ this Lie group admits discrete co-compact subgroups, and therefore the associated compact quotients inherit an abelian almost $3$-contact metric structure with a non-parallel canonical connection. We also study the first homology group of some of these quotients.

\medskip

Let $\{J_1,J_2,J_3\}$ be a hypercomplex structure on the abelian Lie algebra $\R^{4n}$ and let us consider the Lie algebra $\mathfrak{so}(3)$ generated by $\{\xi_1,\xi_2,\xi_3\}$ with Lie brackets given as usual by $[\xi_i,\xi_j]=2\delta \xi_k$ where $(i,j,k)$ is an even permutation of $(1,2,3)$ and $\delta\neq 0$. Then there exists a representation $\rho$ of $\mathfrak{so}(3)$ on $\R^{4n}$ given by
\begin{equation}\label{rho}
 \rho:\mathfrak{so}(3)\to \mathfrak{gl}(4n,\R), \qquad \rho(\xi_i)=\delta J_i, \quad i=1,2,3.
\end{equation}
We denote by $\g_\rho$ the Lie algebra $\g_\rho=\mathfrak{so}(3)\ltimes_\rho \R^{4n}$. Then it is straightforward to verify that $\g_{\rho}$ admits an abelian almost $3$-contact structure with invertible endomorphism $\psi$, and the Lie bracket on $\g_\rho$ is given as in \eqref{beta-delta}.

\begin{proposition}
Let $\{J_1,J_2,J_3\}$ and $\{J_1',J_2',J_3'\}$ be two hypercomplex structures on $\R^{4n}$, with corresponding representations $\rho$, $\rho'$ of $\mathfrak{so}(3)$ on $\R^{4n}$ as in \eqref{rho}. Then the Lie algebras $\g_\rho$ and $\g_{\rho'}$ are isomorphic.
\end{proposition}

\begin{proof}
It is easy to verify that there exist linearly independent vectors $W_1,\ldots,W_n\in\R ^{4n}$ such that $\bigcup_{i=1}^n \{W_i,J_1W_i,J_2W_i,J_3W_i\}$ is a basis of $\R^{4n}$. Analogously, there exist linearly independent vectors $W_1',\ldots,W_n'\in\R ^{4n}$ such that $\bigcup_{i=1}^n \{W_i',J_1'W_i',J_2'W_i',J_3'W_i'\}$ is a basis of $\R^{4n}$. Define $T\in \operatorname{GL}(4n,\R)$ by
\[TW_i=W_i',\quad T(J_1 W_i)=J_1'W_i', \quad T(J_2 W_i)=J_2'W_i',\quad T(J_3 W_i)=J_3'W_i',\]
for all $i$. Therefore $TJ_i=J_i'T$ for all $i$, i.e. $T$ is a intertwining operator for the representations $\rho$ and $\rho'$. Using this property it is straightforward to verify that the linear map $\g_\rho\to \g_{\rho'}$ given by $(\xi,X)\mapsto (\xi, TX)$, for $\xi\in\mathfrak{so}(3)$, $X\in \R^{4n}$, is a Lie algebra isomorphism.
\end{proof}

This proposition says that $\g_{\rho}$ is independent of the hypercomplex structure we begin with, so that we will fix one and will always work with it; the corresponding Lie algebra will be denoted simply by $\g$. Let us identify $\R^{4n}$ with $\H^n$ via
\[ (x_1,y_1,z_1,w_1,\ldots,x_n,y_n,z_n,w_n)\mapsto (x_1+y_1\textbf{i}+z_1\textbf{j}+w_1\textbf{k},\ldots,x_n+y_n\textbf{i}+z_n\textbf{j}+w_n\textbf{k}).   \]
Then we define $J_1=L_{\textbf i},\, J_2=L_{\textbf j},\, J_3=L_{\textbf k}$, where $L_{\textbf i}$ denotes left-multiplication by ${\textbf i}$ on each factor $\H$; analogously for $L_{\textbf j}$, $L_{\textbf k}$. Note that for $n=1$ the hypercomplex structure $\{J_1,J_2,J_3\}$ is given by \eqref{hypercomplex} in the canonical basis of $\R^4$.
		
\medskip

The simply connected Lie group with Lie algebra $\mathfrak{so}(3)$ is $\operatorname{SU}(2)$. We will identify $\operatorname{SU}(2)$ with the Lie group of unit quaternions. Indeed, we can identify $\H$ with the following subspace of complex matrices
\[  \H\equiv \left\{ \begin{pmatrix}
\alpha & \beta \\ -\overline{\beta} & \overline{\alpha}
\end{pmatrix}\mid \alpha,\beta \in\C \right\}, \quad  a+b\textbf{i}+c\textbf{j}+d\textbf{k}\equiv \begin{pmatrix}
a+bi & c+di \\ -c+di & a-bi
\end{pmatrix},   \]
and any unit quaternion gets identified with such a matrix with $|\alpha|^2+|\beta|^2=1$, that is, a matrix in $\operatorname{SU}(2)$. In particular, $\operatorname{SU}(2)$ is diffeomorphic to the $3$-sphere $S^3$.

The representation $\rho:\mathfrak{so}(3)\to \mathfrak{gl}(4n,\R)$ corresponding to the hypercomplex structure defined above can be integrated to a representation $\tilde{\rho}:\operatorname{SU}(2)\to \operatorname{GL}(4n,\R)$, given by
\[  \tilde{\rho}(g)(q_1,\ldots,q_n)= (\delta gq_1,\ldots,\delta gq_n), \quad g\in \operatorname{SU}(2),\, q_r\in\H.   \]
Hence the simply connected Lie group with Lie algebra $\g=\g_{\tilde\rho}$ is $G=\operatorname{SU}(2)\ltimes_{\tilde{\rho}} \H^n$. The product in $G$ is given by
\[ (g,(q_1,\ldots,q_n))(g',(q_1',\ldots,q_n'))=(gg',(q_1+\delta gq_1',\ldots,q_n+\delta gq_n')).   \]

According to Proposition \ref{semidirect} the Lie group $G$ admits a left invariant abelian almost $3$-contact metric structure together with a non-parallel canonical connection.

\medskip

\begin{remark}
It is easy to verify that for $n=1$ the representation $\tilde{\rho}$ of $\operatorname{SU}(2)$ on $\R^4$ is irreducible, while for $n>1$ the representation $\tilde{\rho}$ is reducible, equal to the sum of $n$ copies of the $4$-dimensional one. It follows from \cite{IRS} that the representation $\tilde{\rho}:\operatorname{SU}(2)\to \operatorname{GL}(4,\R)$ is the only $4$-dimensional irreducible real representation of $\operatorname{SU}(2)$, up to equivalence.
\end{remark}

\medskip

\subsection{Compact quotients of $G=\operatorname{SU}(2)\ltimes_{\tilde{\rho}} \H^n$}
We show next that $G$  admits co-compact discrete subgroups. From now on, in this article a co-compact discrete subgroup of $G$ will be called a lattice. We may assume $\delta =1$.

\medskip

We begin by showing that when $n=1$ the group $G=\operatorname{SU}(2)\ltimes_{\tilde{\rho}} \H$ does admit lattices.

For each $m\in\N$, let $\theta_m:=\frac{2\pi}{m}$ and let $G_m$ denote the subgroup of $\operatorname{SU}(2)$ generated by $g_m=\begin{pmatrix} \operatorname{e}^{i\theta_m} & 0 \\ 0 & \operatorname{e}^{-i\theta_m} \end{pmatrix}$, therefore $G_m$ is isomorphic to the cyclic group of order $m$.
Consider now the action of $g_m$ on $\H\equiv \R^4$, given by a $4\times 4$ real matrix $A_m$, associated to $\tilde{\rho}(g_m)$. If $A_m$ is conjugate to an invertible integer matrix $E_m\in\operatorname{GL}(4,\Z)$, i.e. there exists $P_m\in\operatorname{GL}(4,\R)$ such that $P_m^{-1}A_mP_m=E_m$, then $A_m$ preserves the subgroup $P_m\Z^4\subset\R^4$ and hence we can form the semidirect product $\Gamma_m:=G_m\ltimes_{\tilde{\rho}} P_m\Z^4$, which will be a discrete subgroup of $G$. More precisely, we have that \footnote{For two $2\times 2$ matrices $A$ and $B$, we denote $A\oplus B= \begin{pmatrix}
	A & 0 \\ 0 & B 	\end{pmatrix}$.}
\begin{equation}\label{A_m}
 A_m=\begin{pmatrix} \cos \theta_m & -\sin \theta_m  \\ \sin \theta_m & \cos \theta_m \end{pmatrix} \oplus \begin{pmatrix} \cos \theta_m & -\sin \theta_m  \\ \sin \theta_m & \cos \theta_m \end{pmatrix}.
\end{equation}

Analyzing the characteristic polynomial of $A_m$ it is easy to show that $A_m$ is conjugate to an integer matrix if and only if $m\in\{1,2,3,4,6 \}$. For $m=1,2,4$, the matrix $A_m$ is itself integer (therefore we can take $P_m=I_4$), while:
\begin{itemize}
	\item for $m=3$, we can choose $P_3= \begin{pmatrix} 	0 & -\sqrt{3} \\ 2 & -1 	\end{pmatrix} \oplus \begin{pmatrix} 0 & -\sqrt{3} \\ 2 & -1 \end{pmatrix}$;
	\item for $m=6$, we can choose $P_6= \begin{pmatrix} 	2 & 1 \\ 0 & \sqrt{3} 	\end{pmatrix} \oplus \begin{pmatrix} 2 & 1 \\ 0 & \sqrt{3} \end{pmatrix}$.
\end{itemize}
Finally, it is easy to show that for each $m=1,2,3,4,6$, the restriction of the projection $\pi: G\to \Gamma_m\backslash G$ to $\operatorname{SU}(2)\times [0,2]^4$ is surjective. Therefore the quotient manifold $\Gamma_m\backslash G$ is compact, for $m=1,2,3,4,6$.

\medskip

For $n>1$, we can repeat repeat this process in each irreducible subrepresentation $\H$ of $\H^n$, obtaining lattices in $G$ given by $\Gamma_m=G_m\ltimes_{\tilde{\rho}} P_m\Z^{4n}$, for $m=1,2,3,4,6$.

\begin{remark}
	When $m=1$, we have that $G_1=\{1\}$ and $\Gamma_1=\{1\}\times \Z^{4n}$. Consequently, $\Gamma_1\backslash G$ is diffeomorphic to $S^3\times T^{4n}$.
\end{remark}
	
\begin{remark}
	For $n>1$, there are lattices in $G$ of the form $G_m\ltimes_{\tilde{\rho}} P_m\Z^{4n}$ for values of $m$ other than $1,2,3,4,6$. However, for the sake of simplicity, we will keep working with the ones we found previously (i.e., $m=1,2,3,4,6$).
\end{remark}

 \

 Next, we will determine the first homology group (with integer coefficients) of the compact manifolds $M_m:=\Gamma_m\backslash G$, where $\Gamma_m$, $m=1,2,3,4,6$, are the lattices in $G$ constructed above. We obtain as a consequence the first Betti number $b_1(M_m)$ of these manifolds. Since $M_1\cong S^3\times T^{4n}$, whose homology is well known, we will focus on $m=2,3,4,6$.

 Recall that $\Gamma_m=G_m\ltimes_{\tilde{\rho}} P_m\Z^{4n}$, where $G_m$ is a cyclic subgroup of order $m$ and $P_m\in \operatorname{GL}(4n,\R)$ is a matrix such that $E_m:=P_m^{-1}A_mP_m\in \operatorname{GL}(4n,\Z)$, with $A_m$ as in \eqref{A_m}. Hence
 \begin{equation}\label{gamma}
 \Gamma_m\cong \Z_m\ltimes_{E_m} \Z^{4n},
 \end{equation}
 where the action of $1\in \Z_m$ on $\Z^{4n}$ is given by the matrix $E_m$. Note that $E_m$ has order $m$.

 Since $G$ is simply connected, it follows that  $\pi_1(M_m)\cong \Gamma_m$. Thanks to  Hurewicz theorem, we have then that $H_1(M_m,\Z)\cong \Gamma_m/[\Gamma_m,\Gamma_m]$. For the next computations, we will consider $\Gamma_m$ given as in \eqref{gamma}.

 For $(k,u),(l,v)\in \Gamma_m$, we have that $(k,u)(l,v)=(k+l,u+E_m^kv)$ and $(k,u)^{-1}=(-k,-E_m^{-k}u)$, whence
 \begin{align}\label{commutator}
 [(k,u),(l,v)] & = (k,u)(l,v)(k,u)^{-1}(l,v)^{-1} \\
 & = (0,(u-E_m^lu)-(v-E_m^kv)). \nonumber
 \end{align}

 \medskip

 $\bullet$ For $m=2$, we have that $E_2=-I$ and using \eqref{commutator} it is easy to show that
 \[  [\Gamma_2,\Gamma_2]= \{ (0,(u_1,\ldots,u_{4n}))\mid u_i\in 2\Z \text{ for all } i    \}.   \]
 The map $f:\Gamma_2\to \Z_2\oplus(\Z_2)^{4n}$ given by $f(k,(u_1,\ldots,u_{4n}))=(k,[u_1]_2,\ldots,[u_{4n}]_2)$ is a surjective homomorphism with $\operatorname{Ker} f=[\Gamma_2,\Gamma_2]$, thus
 \[  H_1(M_2,\Z)\cong \Gamma_2/[\Gamma_2,\Gamma_2]\cong (\Z_2)^{4n+1}.  \]
 It follows that $b_1(M_2)=0$.

 \medskip

 $\bullet$ For $m=4$, we have that $E_4=\begin{pmatrix} 0 & -1 \\ 1 & 0 \end{pmatrix}\oplus \cdots \oplus \begin{pmatrix} 0 & -1 \\ 1 & 0 \end{pmatrix}$ ($2n$ times), and using \eqref{commutator} it is easy to show that
 \[  [\Gamma_4,\Gamma_4]= \{ (0,(u_1,\ldots,u_{4n}))\mid u_{2i-1}+u_{2i}\in 2\Z \text{ for all } i    \}.   \]
 The map $f:\Gamma_4\to \Z_4\oplus(\Z_2)^{2n}$ given by $f(k,(u_1,\ldots,u_{4n}))=(k,[u_1+u_2]_2,\ldots,[u_{4n-1}+u_{4n}]_2)$ is a surjective homomorphism with $\operatorname{Ker} f=[\Gamma_4,\Gamma_4]$, thus
 \[  H_1(M_4,\Z)\cong \Gamma_4/[\Gamma_4,\Gamma_4]\cong \Z_4\oplus (\Z_2)^{2n}.  \]
 It follows that $b_1(M_4)=0$.

 \medskip

 $\bullet$ For $m=3$, we have that $E_3=\begin{pmatrix} 0 & -1 \\ 1 & -1 \end{pmatrix}\oplus \cdots \oplus \begin{pmatrix} 0 & -1 \\ 1 & -1 \end{pmatrix}$ ($2n$ times), and using \eqref{commutator} it is easy to show that
 \[  [\Gamma_3,\Gamma_3]= \{ (0,(u_1,\ldots,u_{4n}))\mid u_{2i-1}+u_{2i}\in 3\Z \text{ for all } i    \}.   \]
 The map $f:\Gamma_3\to \Z_3\oplus(\Z_3)^{2n}$ given by $f(k,(u_1,\ldots,u_{4n}))=(k,[u_1+u_2]_3,\ldots,[u_{4n-1}+u_{4n}]_3)$ is a surjective homomorphism with $\operatorname{Ker} f=[\Gamma_3,\Gamma_3]$, thus
 \[  H_1(M_3,\Z)\cong \Gamma_3/[\Gamma_3,\Gamma_3]\cong (\Z_3)^{2n+1}.  \]
 It follows that $b_1(M_4)=0$.

 \medskip

 $\bullet$ For $m=6$, we have that $E_6=\begin{pmatrix} 0 & -1 \\ 1 & 1 \end{pmatrix}\oplus \cdots \oplus \begin{pmatrix} 0 & -1 \\ 1 & 1 \end{pmatrix}$ ($2n$ times), and using \eqref{commutator} it is easy to show that
 \[  [\Gamma_6,\Gamma_6]= \{ 0\}\oplus \Z^{4n}.  \]
 In this case it is clear that
 \[  H_1(M_6,\Z)\cong \Gamma_6/[\Gamma_6,\Gamma_6]\cong \Z_6.  \]
 It follows that $b_1(M_6)=0$.

 \

 \begin{example}
 	We provide next an example of a lattice in $G$ which arises from a non-abelian subgroup of $\operatorname{SU}(2)$. Namely, we will consider the quaternionic group $Q_8\subset \operatorname{SU}(2)$, where we identify as before $\operatorname{SU}(2)$ with the unit quaternions. Since $\tilde{\rho}(g)\in \operatorname{GL}(4n,\Z)$ for any $g\in Q_8$, the semidirect product $\Lambda:=Q_8\ltimes_{\tilde{\rho}} \Z^{4n}$ is well defined and it is a discrete subgroup of $G$. One can check that the restriction of the projection $\pi: G\to \Lambda\backslash G$ to $\operatorname{SU}(2)\times [0,1]^{4n}$ is surjective, therefore the quotient manifold $\Lambda\backslash G$ is compact.
 	
 	Concerning the abelianization of $\Lambda$, it is easy to verify that
 	\[ [\Lambda,\Lambda]=\left \{\left(\pm 1,(u_1,\ldots,u_{4n})\right) \mid \sum_{i=1}^{4} u_{4k+i}\in 2\Z, \text{ for } k=0,\ldots,n-1   \right \}.    \]
 	Let us consider the surjective homomorphism $\gamma:Q_8\to \Z_2\oplus \Z_2$ defined by:
 	\[ \gamma(\pm 1)=(0,0),\quad \gamma(\pm \textbf{i})=(1,0),\quad \gamma(\pm \textbf{j})=(0,1),\quad \gamma(\pm \textbf{k})=(1,1). \]
 	Then the map $f:\Lambda\to (\Z_2\oplus \Z_2)\oplus(\Z_2)^{n}$ given by
 	\[ f(g,(u_1,\ldots,u_{4n}))= (\gamma(g),[u_1+u_2+u_3+u_4]_2,\ldots,[u_{4n-3}+u_{4n-2}+u_{4n-1}+u_{4n}]_2)\]
 	is a surjective homomorphism with $\operatorname{Ker} f=[\Lambda,\Lambda]$, thus
 	\[  H_1(\Lambda\backslash G,\Z)\cong \Lambda/[\Lambda,\Lambda]\cong  (\Z_2)^{n+2}.  \]
 	It follows that $b_1(\Lambda\backslash G)=0$.
 \end{example}
	
\

\end{document}